\documentclass[twoside,11pt]{article}

\usepackage[margin=0.9in]{geometry} 
\usepackage{amsmath,amsthm,amssymb,mathrsfs,bbm,titling,mathtools,hyperref,nameref,MnSymbol,stmaryrd}
\usepackage{color}
\usepackage{fancyhdr}
\usepackage[style=numeric,backend=biber]{biblatex}
\let\oldref\ref
\renewcommand{\ref}[1]{(\oldref{#1})}

\newcommand{\N}{\mathbb{N}}
\newcommand{\Z}{\mathbb{Z}}
\newcommand{\C}{\mathbb{C}}
\newcommand{\R}{\mathbb{R}}
\newcommand{\Q}{\mathbb{Q}}

\newcommand{\K}{\mathbb{K}}
\newcommand{\eps}{\varepsilon}

\renewcommand{\cal}{\mathcal}
\renewcommand{\frak}{\mathfrak}

\newtheorem{theorem}{Theorem}[section]
\newtheorem{lemma}[theorem]{Lemma}

\theoremstyle{definition}
\newtheorem{definition}[theorem]{Definition}

\newtheorem{proposition}[theorem]{Proposition}

\theoremstyle{remark}
\newtheorem{remark}[theorem]{Remark}

\numberwithin{equation}{section}

\title{Small cap decoupling for the parabola with logarithmic constant}
\author{Ben Johnsrude}

\begin{document}

	\maketitle
	
	\begin{abstract}
		We note that the subpolynomial factor for the $\ell^qL^p$ small cap decoupling constants for the truncated parabola $\mathbb{P}^1=\{(t,t^2):|t|\leq 1\}$ may be controlled by a suitable power of $\log R$. This is achieved by proving a suitable amplitude-dependent wave envelope estimate, as was introduced in a recent paper of Guth and Maldague to demonstrate a small cap decoupling for the $(2+1)$ cone. The logarithmic loss is reached through a combination of existing techniques for high/low analysis, efficient narrow/broad analysis, and a novel system of rapidly-decaying wave packets.
	\end{abstract}
	
	\section{Introduction}
	
	In this note, we record that the ``wave envelope estimate'' analysis of \cite{GM1} suffices to derive small cap decoupling estimates for functions with Fourier support in the $R^{-1}$-neighborhood of the truncated parabola $\mathbb{P}^1=\{(x,x^2):|x|\leq 1\}$ with constant of the form $(\log R)^{C}$, when combined with previously-established tricks and a novel choice of wave packet functions.
	
	Small cap decouplings were introduced in \cite{DGW}; we recall the formulation here. For large parameters $R>1$, set $\mathcal{N}_{R^{-1}}(\mathbb{P}^1)$ to be the $R^{-1}$-neighborhood of the truncated parabola. Consider a Schwartz function $f:\R^2\to\C$ such that $\text{supp}(\hat{f})\subseteq\mathcal{N}_{R^{-1}}(\mathbb{P}^1)$, where $\hat{\,}$ denotes the Fourier transform. Let $\beta\in[\frac{1}{2},1]$. Partition $\mathcal{N}_{R^{-1}}(\mathbb{P}^1)$ into a collection $\Gamma_\beta(R^{-1})$ of sets $\gamma$, which are the intersections of $\mathcal{N}_{R^{-1}}(\mathbb{P}^1)$ with sets of the form $[c,c+R^{-\beta}]\times\R$; one may note that such $\gamma$ are approximately boxes of dimensions $R^{-\beta}\times R^{-1}$, in the sense that for each $\gamma$ we may find a box $B$ with those dimensions such that $C^{-1}B\subseteq\gamma\subseteq CB$ for a universal constant $C$, where $CB$ and $C^{-1}B$ denote dilation about the center of $B$. Set
	\begin{equation*}
		f_\gamma(x)=\int_\gamma\hat{f}(\xi)e^{2\pi i\xi\cdot x}d\xi
	\end{equation*}
	to be the Fourier projection of $f$ onto $\gamma$. Here and elsewhere all integrals will be with respect to Lebesgue measure. If $p,q\in[1,\infty)$, set $D_{p,q}(R;\beta)$ to be the infimal constant such that
	\begin{equation*}
		\|f\|_{L^p(\R^2)}^p\leq D_{p,q}(R;\beta)\left(\sum_{\gamma\in\Gamma_\beta(R^{-1})}\|f_\gamma\|_{L^p(\R^2)}^q\right)^{p/q}.
	\end{equation*}
	The landmark paper \cite{BD} demonstrated the estimate $D_{p,2}(R;\frac{1}{2})\lesssim_\eps R^\eps$ for all $\eps>0$ and $2\leq p\leq 6$. The authors of \cite{GMW} provided the improved estimate $D_{6,2}(R;\frac{1}{2})\lesssim(\log R)^C$ for a suitable constant $C>0$; the authors of \cite{GLY} sharpened this upper bound to $C_\eps(\log R)^{12+\eps}$ for a bilinear variant over $\Q_p$, implying a matching discrete restriction estimate (over $\R$) with very good logarithmic constant. In another direction, the authors of \cite{DGW} introduced the constants $D_{p,q}(R;\beta)$ for $\beta\in(\frac{1}{2},1]$, and showed that $D_{p,p}(R;\beta)\lesssim_\eps R^{p\beta(\frac{1}{2}-\frac{1}{p})+\eps}$ for all $\eps>0$ and $2\leq p\leq 2+\frac{2}{\beta}$ (Theorem 3.1). Each of these bounds is sharp up to the subpolynomial factors.
	
	Our goal will be to show the following:
	\begin{theorem}[Small cap decoupling with logarithmic losses]\label{smallcap} Let $p,q\geq 1$ satisfy $\frac{3}{p}+\frac{1}{q}\leq 1$, $R>2$, and $\beta\in[\frac{1}{2},1]$. Then the small cap decoupling constant satisfies
		\begin{equation}\label{smd}
			D_{p,q}(R;\beta)\lesssim(\log R)^{18+3p}\Big(R^{\beta(p-\frac{p}{q}-1)-1}+R^{p\beta(\frac{1}{2}-\frac{1}{q})}\Big).
		\end{equation}
		
	\end{theorem}
	
	This formulation of the decoupling estimate, with instead a factor of $C_\eps R^\eps$ in place of the logarithmic factor, was originally proven in \cite{FGM} (Corollary 5). For each triple $(p,q,\beta)$, the dominating term on the right-hand side in \ref{smd} may be realized by a particular choice of $f$ with large $R$, as demonstrated in \cite{FGM} (Section 2), up to the subpolynomial factor. Thus the power-law terms are each separately sharp in the regime where they dominate. 
 
    In \cite{B} (Remark 2), it was demonstrated using number theory methods that $D_{6,2}(R;\frac{1}{2})\gtrsim(\log R)R$. It is not currently known if there is any other $p,\beta$ with $2\leq p\leq 2+\frac{2}{\beta}$ such that the subpolynomial factor is unbounded in $R$.
	
	Our estimate \ref{smd} is derived by first proving a version of an auxiliary wave envelope estimate, which is precisely stated in Theorem \ref{squarefunction}. We will write $|S|$ to denote the Lebesgue measure of sets $S$.
	
	\begin{theorem}[Wave envelope estimate]\label{squarefunction} There exists a constant $E_2>0$ such that the following holds for all $R\gg 1$. Let $f:\R^2\to\C$ be Schwartz with Fourier support in $\mathcal{N}_{R^{-1}}(\mathbb{P}^1)$. Then, for any $\alpha>0$,
		\begin{equation*}
			\alpha^4|\{x:|f(x)|>\alpha\}|\lesssim(\log R)^{20}\sum_{\substack{R^{-1/2}\leq s\leq 1\\s\,\mathrm{ dyadic}}}\sum_{\tau:\ell(\tau)=s}\sum_{U\in\mathcal{G}_\tau}|U|^{-1}\|S_Uf\|_{L^2(\R^2)}^4.
		\end{equation*}
		
	\end{theorem}
	
	Here we use the following notation: $U_{\tau,R}$ is a rectangle of dimensions $R\times sR$, with long edge in the direction of the normal vector to $\mathbb{P}^1$ at the center of $\tau$, centered at $0$; the set $\mathcal{G}_\tau$ is (essentially) the subset of the tiling of $\R^2$ by translated copies of $U_{\tau,R}$ for which the following holds:
	\begin{equation}\label{superest}
		C(\log R)^8|U|^{-1}\int_U\sum_{\theta\subseteq\tau}|f_{\theta}|^2\geq\frac{\alpha^2}{(\#\tau)^2}
	\end{equation}
	for suitable choice of $C>0$. Here $\#\tau$ denotes the number of $\tau$ of a particular length for which $f_\tau\not\equiv 0$. Lastly, we use $S_Uf$ to denote the restricted square function $\left.(\sum_{\theta\subseteq\tau}|f_\theta|^2)\right|_U$; one may observe that the quantities $s$ and $R$ may be read off of the dimensions of $U$, and $\tau$ is then uniquely determined from the direction of $U$'s long edge, so this definition is well-formed.
	
	Wave envelope estimates were introduced in \cite{GWZ} for the purpose of proving the reverse square function estimate for the cone in $\R^3$ (Theorem 1.3). In \cite{GM1}, these wave envelope estimates were refined to include only those envelopes corresponding to ``high-amplitude'' components of the various square functions. The latter paper demonstrated that the wave envelope estimate could also be used to derive the small cap results of \cite{FGM}. Our argument closely follows that of \cite{GM1}, but with various technical refinements to facilitate a logarithmic constant in the wave envelope estimate (e.g. a gentler sequence of scales $R_k$).
	
	We make use of the following notation. For $A,B>0$, we say $A\lesssim B$ if $|A|\leq CB$ for a suitable constant $C$ which may vary from line to line, which does not depend on any variable parameters in the problem unless explicitly indicated. We also write $A\sim B$ if $A\lesssim B$ and $B\lesssim A$. The expression $O(B)$ will be used to denote a quantity which is $\lesssim B$. We also note from the outset that we slightly redefine the notation $\strokedint$ to something better suited to our purposes than its usual meaning; see the pruning section below.
	
	Throughout the paper, given a parallelogram $P$, we will write $c_P$ for the center of $P$. For a scalar $\lambda>0$, we will write $\lambda P$ for the box with the same center $c_P$ but with sidelengths increased by the factor $\lambda$. We will also use an asterisk $*$ to denote a dual of parallelograms, that is, $P^*=A^{-\top}([-\frac{1}{2},\frac{1}{2}]^2)$ when $P=A([-\frac{1}{2},\frac{1}{2}]^2)$.

    Subsequent to the announcement of this result, the author proved \cite{johnsrude2024high} the same results with the underlying field $\R$ replaced by general non-Archimedean local fields $\K$ of characteristic $\neq 2$. The argument there closely follows the argument here, but the non-Archimedean flavor permits one to omit many technical arguments. Consequently, the reader may find it useful to refer to that manuscript to understand the essence of the argument.
	
	The remainder of the paper is organized as follows. In Section \ref{presqfunctionproof}, we first give an overview of the argument, then construct the wave packets, then state the pruning and technical lemmas needed in the proof of Theorem \ref{squarefunction}. In Section \ref{sqfunctionproof}, we prove Theorem \ref{squarefunction}. In Section \ref{smallcapproof}, we show that Theorem \ref{squarefunction} proves Theorem \ref{smallcap}. In Section \ref{section:appendix}, we prove the technical lemmas from Section \ref{presqfunctionproof}.

    \subsection{Acknowledgements}

    The author would like to thank Terence Tao and Hong Wang for many helpful comments and suggestions. The author would like to thank Jaume de Dios Pont for suggesting the use of Gevrey-class functions in place of Gaussians for the purpose of defining rapidly-decaying wave packets. The author would also like to thank Zane Kun Li and Po-Lam Yung for providing helpful feedback on a pervious draft of this manuscript.
	
	\section{Infrastructure for proving Theorem \ref{squarefunction}}\label{presqfunctionproof}
	
	\subsection{Overview of the argument}

    We first indicate the basic obstruction in proving logarithmically-fine decoupling estimates $D_p(\delta)\lesssim(\log\delta^{-1})^{O(1)}$ for the parabola. The classic Bourgain-Demeter scheme relies on a multiscale decoupling estimate of the form
    \begin{equation}\label{ineq:bd_scheme}
        D_{6,2}(\delta,1/2)\leq D_{6,2}(\delta^{1/2},1/2)D_{6,2}(\delta^{1/4},1/2)\cdots D_{6,2}(\delta^\eps,1/2).
    \end{equation}
    It is clear that \ref{ineq:bd_scheme} satisfies a pleasant dimensional consistency with power-law bounds $D_{6,2}(\delta,1/2)\leq \delta^{-\eta}$. Such an inequality ``almost'' proves the Bourgain-Demeter estimate $D_{6,2}(\delta,1/2)\leq C_\eps\delta^{-\eps}$ for all $\eps>0$: a small improvement over trivial bounds on the decoupling constant at some scale suffices to combine with induction-on-scales technology to prove that estimate. By contrast, it is clear that \ref{ineq:bd_scheme} clearly fails to allow one to upgrade logarithmic bounds at scales $\rho>\delta$ to a logarithmic bound at $\delta$.

    More seriously, logarithmic errors are known to be possible at each individual scale: \cite{B} proved the inequality $D_{6,2}(\delta,1/2)\gtrsim\log \delta^{-1}$ at exponent $6$. If each factor of \ref{ineq:bd_scheme} satisfies that lower bound, then the product is unacceptably large. Consequently, we need to demonstrate that each particular datum $f$ can represent the ``bad arrangement'' found in \cite{B} at most over $O(1)$ distinct scales $\rho$. This is accomplished by decomposing $f$ additively as $f=\sum_nf_n$, where each $f_n$ can only have interesting behavior at $O(1)$ distinct scales.
	
	We now recall the general intuition behind the shape of the right-hand side of Theorem \ref{squarefunction}, without considering the amplitude dependence. We will first be concerned with the decomposition only as it is constructed in \cite{GM1}, and later indicate where more efficient methods are indicated towards the end of this subsection.
 
    Consider a Schwartz function $f:\R^2\to\C$ with Fourier support contained in $\mathcal{N}_{R^{-1}}(\mathbb{P}^1)$. The $L^4$ square function estimate for $\mathbb{P}^1$ implies that
	\begin{equation*}
		\int|f|^4\lesssim\int\big|\sum_\theta|f_\theta|^2\big|^2.
	\end{equation*}
	By Plancherel,
	\begin{equation*}
		\int\big|\sum_\theta|f_\theta|^2(x)\big|^2=\int\big|\widehat{\sum_\theta|f_\theta|^2}(\xi)\big|^2.
	\end{equation*}
	We study the latter integral by considering the contributions from different dyadic bands of $|\xi|$. Since each $f_\theta$ has Fourier support contained in the cap $\theta$ of size $\sim R^{-1/2}\times R^{-1}$, the support of the latter integral is contained in the ball of radius $2R^{-1/2}$ centered at the origin, so we only need to consider frequency contributions below this magnitude.
	
	On the other hand,
	\begin{equation*}
		\int_{|\xi|<R^{-1}}\Big|\widehat{\sum_\theta|f_\theta|^2}(\xi)\Big|^2\lesssim \int\Big|\sum_\theta|f_\theta|^2*(R^{-2}w_{B_R})\Big|^2
	\end{equation*}
	for a suitable weight $w_{B_R}$ which is $\sim 1$ on $B_R$ and rapidly decays outside of $B_R$; if we write the latter integral as a sum of integrals over cubes $Q_R$,
	\begin{equation*}
		\int\Big|\sum_\theta|f_\theta|^2*(R^{-2}w_{B_R})\Big|^2=\sum_{Q_R}\int_{Q_R}\Big|\sum_\theta|f_\theta|^2*(R^{-2}w_{B_R})\Big|^2.
	\end{equation*}
	Since $B_R$ is a square of sidelength $R$, the convolution is approximately constant on such $Q_R$. Thus
	\begin{equation*}
		\sum_{Q_R}\int_{Q_R}\Big|\sum_\theta|f_\theta|^2*(R^{-2}w_{B_R})\Big|^2\lesssim\sum_{Q_R}|Q_R|^{-1}\left(\int W_{Q_R}\sum_\theta|f_\theta|^2\right)^2
	\end{equation*}
	for suitable weights $W_{Q_R}$ which are approximate cutoffs to the set $Q_R$. Thus
	\begin{equation*}
		\int_{|\xi|<R^{-1}}\Big|\widehat{\sum_\theta|f_\theta|^2}(\xi)\Big|^2\lesssim\sum_{Q_R}|Q_R|^{-1}\left(\int W_{Q_R}\sum_\theta|f_\theta|^2\right)^2,
	\end{equation*}
	which is one of the summands on the right-hand side of \ref{squarefunction}.
	
	More generally, if we consider integrals of the form
	\begin{equation*}
		\int_{|\xi|\sim r}\Big|\widehat{\sum_\theta|f_\theta|^2}(\xi)\Big|^2,\quad R^{-1}<r\leq R^{-1/2},
	\end{equation*}
	then we may instead make use of the approximate orthogonality of the families $\{\sum_{\theta\subseteq\tau}|f_\theta|^2\}_{\ell(\tau)=\frac{1}{rR}}$ on the annulus $\{|\xi|\sim r\}$; notice that, by finite overlap,
	\begin{equation*}
		\int_{|\xi|\sim r}\Big|\widehat{\sum_\theta|f_\theta|^2}(\xi)\Big|^2\lesssim\sum_{\tau}\int_{|\xi|\sim r}\Big|\widehat{\sum_{\theta\subseteq\tau}|f_\theta|^2}(\xi)\Big|^2,
	\end{equation*}
	and that the functions
	\begin{equation*}
		\sum_{\theta\subseteq\tau}|f_\theta|^2*\chi_{\sim r}^\vee
	\end{equation*}
	are approximately constant on sets of the form $U\|U_{\tau,R}$, where $\chi_{\sim r}$ is a smooth cutoff to the annulus $|\xi|\sim r$. Thus, as above,
	\begin{equation*}
		\int_{|\xi|\sim r}\Big|\widehat{\sum_{\theta\subseteq\tau}|f_\theta|^2}(\xi)\Big|^2\lesssim\sum_{U\|U_{\tau,R}}|U|^{-1}\left(\int W_U\sum_{\theta\subseteq\tau}|f_\theta|^2\right)^2,
	\end{equation*}
	which is also of the right shape for our theorem. We have essentially validated the ``wave envelope bound''
    \begin{equation*}
        \int\Big|\sum_{d(\tau)=\frac{1}{rR}}|f_\tau|^2*\chi_{>r}^\vee\Big|^2\lessapprox\sum_{s\leq r}\sum_{d(\tau)=\frac{1}{sR}}\sum_{U\|U_{\tau,R}}|U|^{-1}\int\Big(\sum_{\theta\subseteq\tau}|f_\theta|^2*\chi_{\leq s}^\vee\Big)^2,
    \end{equation*}
    which dominates the high part of a square function for arbitrary $f$ by an expansion into wave envelopes.
	
	We may observe from the preceding calculation that we would have proved Theorem \ref{squarefunction} if, for each $\tau$ and each $U\|U_{\tau,R}$, we had the estimate
	\begin{equation*}
		C(\log R)^8|U|^{-1}\int_U\sum_{\theta\subseteq\tau}|f_\theta|^2\geq\frac{\alpha^2}{(\#\tau)^2},
	\end{equation*}
	or else $S_Uf$ is negligible, say $O(R^{-1000})$. It is therefore natural to split $f$ into pruned pieces for which the non-negligible $S_Uf$ satisfy the ``good'' estimate above, at various scales. Our prunings, following \cite{GM1}, will therefore be written as follows:
	\begin{equation*}
		\begin{split}
			f&=f_N+f^\mathcal{B}\\
			f_N&=f_{N-1}+f_N^\mathcal{B}\\
			f_{N-1}&=f_{N-2}+f_{N-1}^\mathcal{B}\\
			&\quad\cdots\\
			f_2&=f_1+f_2^\mathcal{B}
		\end{split}
	\end{equation*}
	where $f_m$ is given by trivializing the contributions $S_Uf$, $U\|U_{\tau,R}$, $d(\tau)\lesssim(\log R)^{-m}$, for which \ref{superest} fails.
	
	To illustrate, the first phase of pruning is as follows. Take the wave packet expansion of $f$ at scale $R$, say
	\begin{equation*}
		f\approx\sum_\theta\sum_{T\in\mathbb{T}_\theta}\psi_Tf_\theta,
	\end{equation*}
	and define $f_N$ to be
	\begin{equation*}
		f_N=\sum_\theta\sum_{T\in\mathbb{T}_\theta'}\psi_Tf_\theta,
	\end{equation*}
	where $\mathbb{T}_\theta'$ is the set of $T$ for which
	\begin{equation*}
		C_{\mathfrak{p}}(\log R)^8|T|^{-1}\int_T|f_\theta|^2\gtrsim\frac{\alpha^2}{(\#\theta)^2}
	\end{equation*}
	for a suitable pruning constant $C_{\mathfrak{p}}$. If we apply the $L^4$ square function estimate/Plancherel/dyadic pigeonholing argument outlined above to $f_N$, then the contribution of the integral along $|\xi|\sim r$ of $f_N$ will be acceptable for Theorem \ref{squarefunction} when $r\gtrapprox R^{-1/2}$.
	
	However, the other annular integrals will involve wave envelopes of other dimensions which have not yet been pruned, and it will be necessary to consider deeper prunings. In particular, if we decompose $f_N=f_{N-1}+f_N^\mathcal{B}$ by defining
	\begin{equation*}
		f_{N-1}=\sum_{\theta}f_{N-1,\theta},
	\end{equation*}
	with $f_{N-1,\theta}$ equal to the sum of the wave envelopes of scale $\sim 2R^{1/2}\times R$ with appropriately high amplitude square functions, then more of the integrals of $f_{N-1}$ will be acceptable; on the other hand, since $f_N^\mathcal{B}$ is high-amplitude on small wave packets and low-amplitude on larger wave packets, it must be that $f_N^\mathcal{B}$ is dominated by high-frequency contribution (as otherwise low-dominance would imply sufficient local constancy to reach a contradiction).
	
	Proceeding inductively, we replace $f$ by a sum of $N$ functions 
    \begin{equation*}
        f=f_1+\sum_{m=2}^Nf_m^\mathcal{B},
    \end{equation*}
    where the ``bad'' functions $f_m^\mathcal{B}$ have acceptable high-frequency contributions and are also dominated by those contributions, and where the lowest function $f_1$ satisfies the wave envelope estimate by construction.

    We now indicate what is needed to obtain a logarithmic loss in Theorem \ref{squarefunction}. We refine the argument of \cite{GM1} by applying a modified broad/narrow argument and a modified pigeonholing, which are chosen to avoid superlogarithmic losses. We also choose a longer and gentler sequence of scales ($R_{k+1}/R_k= O(1)$ as opposed to $R^\eps$) to minimize the cost of applying the high lemmas. Each of these have appeared elsewhere in the literature before; for example, the broad/narrow argument is adapted from \cite{GMW}.

	The primary technical advantage in the current work is the use of wave packets with near-exponential decay, which permits one to improve Schwartz-type decay to decay of the form $e^{-|x|^{1-\eps}}$, while preserving compact support on the Fourier side; the details are offered in subsection \ref{gevpackets} below. Such decay on the spatial side is sufficient to prevent super-logarithmic losses in our setting, particularly when estimating the interference of parallel wave packets via Cauchy-Schwarz.
 
    We indicate briefly how these gains are manifested. Given a partition of unity $\{\psi_T\}_T$ made up of wave packets of dimensions $R^{1/2}\times R$, for which each $\psi_T$ is concentrated near $T\subseteq\R^2$ and $\widehat{\psi_T}$ is supported in $T^*$, we may fix some $x\in\R^2$ and unique $T\ni x$. The question of interference may be summarized as follows: for which constant $M$ does it hold that
    \begin{equation*}
        \sum_{T'\|T:T'\cap MT=\emptyset}\psi_{T'}(x)\lesssim R^{-1000}\quad ?
    \end{equation*}
    Knowing only a Schwartz decay on $\psi_T(x)=\psi(R^{-1/2}x_1,R^{-1}x_2)$ (say), $|\psi(x)|\lesssim_D|x|^{-D}$, we may only conclude that $M\approx_\eps R^\eps$ suffices. If $\psi$ decays at a Gaussian rate (which is inconsistent with the compact Fourier support condition), then we may take $M=O(\log R)$. If $\psi$ decays at a slower rate $|\psi(x)|\leq e^{-c|x|^{1/2}}$, we may take $M=O((\log R)^3)$, which suffices for our purposes. Due to technical obstructions in further arranging for $\psi$ to be positive, we slightly weaken the exponent $\frac{1}{2}$ to $\frac{1}{2.2}$.
 
    The authors of \cite{GMW} handled this issue by appealing to wave packets defined by Gaussian weights, which possess the technical difficulty of having noncompact Fourier support. Note too that, by analyticity, the decay $e^{-c|x|^{1/2}}$ could not be improved to $e^{-c|x|}$.

    \subsection{Construction of the rapidly-decaying wave packets}\label{gevpackets}
    
    We will need a partition of unity composed of wave packets which decay almost exponentially, and which have compact Fourier support. We will make critical use of the \textit{Gevrey class} $G^s(\R^n)$ of functions, which may precisely be defined as
    \begin{equation*}
        G^s(\R^n)=\{g\in C^\infty(\R^n;\C):\exists C>0\text{ s.t. }|\partial^{\alpha}f(x)|\leq C^{|\alpha|+1}|\alpha!|^s\,\,\forall\alpha\in\Z_{\geq 0}^n,\forall x\in\R^n\}
    \end{equation*}
    Here $1\leq s<+\infty$, and for a multi-index $\alpha\in\Z_{\geq 0}^n$ we set $|\alpha|=\sum_{j=1}^n\alpha_j,\alpha!=\alpha_1!\cdots\alpha_n!$.

    One may readily observe that $G^s(\R^2)$ is a vector space and is closed under pointwise multiplication and differentiation. The class $G^1(\R^2)$ coincides with the class $A(\R^2)$ of functions which locally are locally given by convergent power series in the variables $x,y$. On the other hand, the classes $G^s(\R^2)$ ($s>1$) are strictly broader.

    We show a convenient construction for our purposes. Let $a_1\leq a_2\leq \ldots$ be an increasing sequence of positive reals whose reciprocals are summable: $a=\sum_ja_j^{-1}<\infty$. Let $H_{a_j}$ be the auto-convolution of a suitable $\mathrm{rect}$ function with itself:
    \begin{equation*}
        H_{a_j}=\Big(a_j1_{[-\frac{1}{2a_j},\frac{1}{2a_j}]}\Big)*\Big(a_j1_{[-\frac{1}{2a_j},\frac{1}{2a_j}]}\Big).
    \end{equation*}
    Then the sequence of functions $u_k=H_{a_1}*\cdots*H_{a_k}$ converges to a smooth $u$ supported in $[-a,a]$, which moreover satisfies the derivative estimates
    \begin{equation*}
        \|u^{(k)}\|_\infty\leq 2^k\prod_{j=1}^{k+1}a_j^2.
    \end{equation*}
    See \cite{hormander1963linear}, pages 19-20. Moreover, one may show that $\hat u\geq 0$ everywhere. Choose
    \begin{equation*}
        a_j=C\eps^{-1}(j+1)^{1.1}.
    \end{equation*}
    It follows that $u\in G_0^{2.2}(\R)$. By tuning $C$, we may assume that $\mathrm{supp}(u)\subseteq[-\frac{1}{2},\frac{1}{2}]$. By scaling $u$, we may take $u(0)=1$. We write $\rho_0=u\otimes u\in G^{2.2}(\R^2)$. Then $\rho_0$ satisfies the following properties:
		\begin{itemize}
			\item[(a)] $\text{supp}(\rho_0)\subseteq[-\frac{1}{2},\frac{1}{2}]^2$.
			\item[(b)] $\rho_0^\vee\geq 0$ everywhere.
			\item[(c)] $\rho_0(0)=1$.
		\end{itemize}
    
    Write $G_0^s(\R^2)=G^s(\R^2)\cap C_0^\infty(\R^2)$ for the Gevrey-class functions of compact support. One of the critical properties of this class is the following:

    \begin{theorem}[Theorem 1.2(i) of \cite{HR}] Assume $g\in G_0^s(\R^2)$. Then there exist $C,\epsilon>0$ such that
    \begin{equation*}
        |\hat{g}(\xi)|\leq C\exp\left(-\epsilon|\xi|^{1/s}\right)
    \end{equation*}
    for all $\xi\in\R^2$.
    \end{theorem}

    Since $\rho_0\in G_0^{s+\eps}(\R^2)$, we may append the following:
    \begin{itemize}
        \item[(d)] $|\rho_0^\vee(x)|\leq Ce^{-c|x|^{1/2.2}}$,
    \end{itemize}
    for suitable $c,C$.

    The preceding is standard in microlocal analysis of PDEs, see e.g. \cite{HR} for an overview of the methodology and \cite{BLR} for a sample application to scattering theory. The classes were introduced in \cite{G}. They generally serve as useful interpolants between analytic functions and smooth functions.

    We make use of this function $\rho_0$ to construct a suitable partition of unity:
 
	\begin{definition}[Sufficiently rapid cutoffs] Fix a small constant $\epsilon_0>0$. Let $\rho_0$ be any function satisfying the properties (a)-(d) above. For each parallelogram $T$, let $\rho_T=\rho_{0}\circ R_T$, where $R_T$ is an affine transformation that scales and rotates $T$ to $[-\frac{1}{2},\frac{1}{2}]^2$. Define also $\psi_T(x)=|T|^{-1}\rho_{T^*}^\vee(x-c_T)$, where $c_T$ is the center of $T$.

    \end{definition}
  
        Observe that $\rho_T(c_T)=1$ and $\rho_T=0$ outside of $T$. Observe from the outset that $\|\rho_T^\vee\|_1=\|\rho_{0}^\vee\|_1=O(1)$ by change-of-variable.

        \begin{proposition}[Existence of a Gevrey-class partition of unity] Let $T$ be a parallelogram and $\{U\|T\}$ be the fundamental tiling of $\R^2$ by translates $U$ of $T$. Then the functions $\{\psi_U:U\|T\}$ form a partition of unity in $\R^2$.
            
        \end{proposition}
        \begin{proof}

        Note that the set of centers $\{c_U:U\|T\}$ form a lattice, and the centers $\{c_V:V\|T^*\}$ form the dual lattice. By the Poisson summation formula,
		\begin{equation*}
			\sum_{V\|T^*}\rho_{T^*}^\vee(x-c_V)=\frac{1}{|T|}\sum_{U\|T}e^{2\pi ix\cdot c_U}\rho_{T^*}(c_U)=\frac{1}{|T|}.
		\end{equation*}
            
        \end{proof}

        The preceding will be used to decompose our function $f$ below.
	
	\subsection{Initial notation-setting}
	
	We begin by reproducing some of the language of \cite{GM1}, with minor modifications. Fix arbitrary $\alpha>0$, and $R\in 7^{2^\N}$ sufficiently large; we will occasionally assume that $R$ is large enough that $\log\log R$ exceeds a universal constant. Throughout the paper, we will use $B_R$ to denote the ball of radius $R$ centered at $0$. Let $U_\alpha=\{x\in B_R:|f(x)|>\alpha\}$.
	
	We will need a sequence of scales. Let $N$ be the least integer greater than or equal to $\frac{1}{2}\frac{\log R}{\log 7}$. Let $R_k:=7^{k}$ for $k=0,\ldots, N-1$, and define $R_N:=R^{1/2}$.
	
	Next, let $\{\theta\}$ be a partition of $\mathcal{N}_{R^{-1}}(\mathbb{P}^1)$ by approximate $R^{-1/2}\times R^{-1}$ rectangles, and similarly let $\{\tau_k\}$ be a partition of $\mathcal{N}_{R_k^{-1}}(\mathbb{P}^1)$ by approximate $R_k^{-1}\times R_k^{-2}$ rectangles; here and throughout the paper, the notations $\tau_N$ and $\theta$ are interchangeable. We assume that for $k<k'$ and each choice of $\tau_k,\tau_{k'}$ we either have $\tau_{k'}\subseteq\tau_{k}$ or $\tau_{k}\cap\tau_{k'}=\emptyset$. We also write $\tau_0$ for the full $\mathcal{N}_{R^{-1}}(\mathbb{P}^1)$. Furthermore, for each $1\leq k\leq N$ and $\tau_k$, we will write $\tilde\tau_k$ for the union of $\tau_k$ and its immediate neighbors within $\tau_{k-1}\supseteq\tau_k$. If $k=0$, we write $\tilde\tau_0=\tau_0$.
	
	By scaling, it will suffice to consider the case when $\max_\theta\|f_\theta\|_\infty=1$; since we are bounding $|U_\alpha|$, we will assume also that $\alpha\leq R^{1/2}$. By considering the summand on the right-hand side of the inequality in Theorem \ref{squarefunction} corresponding to $s=1$, it suffices to consider the case $\alpha\geq 1$.
	
	For each point $p\in\mathbb{P}^1$, let $\mathbf{t}_p$ be the tangent vector to $\mathbb{P}^1$ at $p$ pointing in the positive-$x$ direction. Similarly, write $\mathbf{n}_p$ for the normal vector to $\mathbb{P}^1$ at $p$ pointing in the positive-$y$ direction.
	
	For each fixed $\tau_k$, we will let $U_{\tau_k,R}$ be a rectangle of dimensions $(R/R_k)\times R$ with long side parallel to $\mathbf{n}_{c_{\tau_k}}$. Fix also a tiling of $\R^2$ by translates $U$ of $U_{\tau_k,R}$; we will denote the relationship between $U$ and $U_{\tau_k,R}$ by $U\|U_{\tau_k,R}$, so that the tiling just described is the set $\{U\|U_{\tau_k,R}\}$.
	
	We will relate different square functions by means of analyzing their high- and low-frequency components. To this end, set $\varphi$ to be a smooth nonnegative radial bump function on $\R^2$ such that $\varphi(\xi)=1$ on $|\xi|\leq 1$ and $\varphi(\xi)=0$ on $|\xi|\geq 2$. For each $r>0$, we define the cutoff functions
	\begin{equation*}
		\eta_{\leq r}(\xi)=\varphi(r^{-1}\xi),\quad \eta_{>r}(\xi)=\varphi(\xi)-\varphi(r^{-1}\xi),\quad \eta_{\sim r}(\xi)=\varphi(r^{-1}\xi)-\varphi(2r^{-1}\xi).
	\end{equation*}
	Note in particular that $\eta_{\leq r}(\xi)=1$ on $|\xi|\leq r$ and $\eta_{\leq r}(\xi)=0$ on $|\xi|>2r$, and $\eta_{>r}(\xi)=1$ on $2r<|\xi|\leq 1$ and $\eta_{>r}(\xi)=0$ on $|\xi|\in(0,r)\cup(2,\infty)$.
	
	Next, for $U\|U_{\tau_k,R}$ let $W_U$ denote the composition $(W\circ T_{\tau_k})(x-c_U)$, where
	\begin{equation*}
		W(x,y):=\frac{1}{(1+|x|^2)^{100}(1+|y|^2)^{100}},
	\end{equation*}
	and $T_{\tau_k}$ is the linear transformation which rotates $2U_{\tau_k,R}$ to $[-R/R_k,R/R_k]\times[-R,R]$ and then rescales to $[-1,1]^2$. We define $\strokedint_Ug:=|U|^{-1}\int gW_U$ for arbitrary $g$. Since $W$ decays polynomially, we may assume $\psi_U\lesssim W_U$ for every choice of $U$.
	
	Next, for each $k$, let $w_k$ be the weight
	\begin{equation*}
		w_k(x)=\frac{c}{(1+|x|^2R_k^{-1})^{10}}
	\end{equation*}
	with $c$ chosen so that $\|w_k\|_1=1$.

	\subsection{Pruning}
	
	For suitable constant\footnote{The size of $C_{\mathfrak{p}}$ is only constrained by the proof of Lemma \ref{highdomlemma}.} $C_{\mathfrak{p}}>0$, we define the pruned set $\mathcal{G}_\theta$ associated to $\theta$ as follows.
	
	\begin{definition} Set\footnote{Recall from above that we have repurposed the symbol $\strokedint_U\cdot$ to mean $|U|^{-1}\int W_U\cdot$}
		\begin{equation*}
			\mathcal{G}_\theta:=\left\{U\|U_{\theta,R}:C_{\mathfrak{p}}(\log R)^{8}\strokedint_U|f_\theta|^2\geq\frac{\alpha^2}{(\#\theta)^2}\right\}.
		\end{equation*}
		Define the pruned functions as
		\begin{equation*}
			f_{N,\theta}:=\sum_{U\in\mathcal{G}_\theta}\psi_Uf_\theta,\quad f_N:=\sum_\theta f_{N,\theta}.
		\end{equation*}
		For $k<N$ and each $\tau_k$, define
		\begin{equation*}
			\mathcal{G}_{\tau_k}:=\left\{U\|U_{\tau_k,R}:C_{\mathfrak{p}}(\log R)^{8}\strokedint_U\sum_{\theta\subseteq\tau_k}|f_{k+1,\theta}|^2\geq\frac{\alpha^2}{(\#\tau_k)^2}\right\}
		\end{equation*}
		and
		\begin{equation*}
			f_{k,\theta}:=\sum_{U\in\mathcal{G}_{\tau_k}}\psi_Uf_{k+1,\theta}\,\,(\text{where }\tau_k\supseteq\theta)\quad\text{and}\quad f_k=\sum_{\theta}f_{k,\theta}.
		\end{equation*}
		We set also $f_k-f_{k-1}=:f_k^{\mathcal{B}}$, and $f_{k,\theta}^{\mathcal{B}}=\sum_{U\not\in\mathcal{G}_{\tau_{k-1}}}\psi_Uf_{k,\theta}$, where $\theta\subseteq\tau_{k-1}$. If $k'\leq k$, then set $f_{k,\tau_{k'}}^{\mathcal{B}}=\sum_{\theta\subseteq\tau_{k'}}f_{k,\theta}^{\mathcal{B}}$.
	\end{definition}
	
	The following estimates will be needed:
	\begin{lemma}[Pruning lemmas]\label{pruninglemmas} The pruned functions satisfy the following:
		
		\begin{enumerate}
			\item[(a)] $f_N=f_1+\sum_{m=2}^Nf_m^{\mathcal{B}}$.
			\item[(b)] $|f_{k,\theta}|\leq|f_{k+1,\theta}|\leq|f_\theta|$.
			\item[(c)] $\mathrm{supp}\big(\widehat{f_{k,\theta}}\big)\subseteq 2(N-k)\theta$ for all $\theta$.
		\end{enumerate}
	\end{lemma}
	\begin{proof}
		(a): This is just the calculation
		\begin{equation*}
			f_1+\sum_{m=2}^Nf_m^\mathcal{B}=f_1+\sum_{m=2}^N(f_m-f_{m-1})=f_N.
		\end{equation*}
		
		(b): Since $\sum_{U\in\mathcal{G}_{\tau_k}}\psi_U\leq 1$, it follows that
		\begin{equation*}
			|f_{k,\theta}|=|f_{k+1,\theta}|\big|\sum_{U\in\mathcal{G}_{\tau_k}}\psi_U\big|\leq |f_{k+1,\theta}|,
		\end{equation*}
		and similarly
		\begin{equation*}
			|f_{N,\theta}|=|f_\theta|\sum_{U\in\mathcal{G}_{\tau_N}}\psi_U\leq |f_\theta|.
		\end{equation*}
		
		(c): We first consider the case $k=N$. For each $\theta$ and $U\|U_{\theta,R}$,
		\begin{equation*}
			\begin{split}
				\widehat{\psi_Uf_\theta}(\xi)=\int\widehat{\psi_U}(\eta)\hat{f}_\theta(\xi-\eta)d\eta,
			\end{split}
		\end{equation*}
		which vanishes when there does not exist $\eta\in 2U^*\subseteq B(0,2R^{-1})$ such that $\xi-\eta\in\theta$, i.e. when $\xi\not\in \mathcal{N}_{2R^{-1}}\theta$. Thus $f_{N,\theta}$ has Fourier support in $\theta+B(0,2R^{-1})$.
		
		More generally, the same calculation gives
		\begin{equation*}
			\operatorname{supp}(\hat{f}_{k,\theta})\subseteq\theta+2\sum_{j=0}^{N-k}U_{\tau_{N-j}}^*,
		\end{equation*}
		where $\tau_{N-j}$ is the cap of size $R_{N-j}^{-1}\times R_{N-j}^{-2}$ containing $\theta$. In particular,
		\begin{equation*}
			\operatorname{supp}(\hat{f}_{k,\theta})\subseteq 2(N-k)\theta,
		\end{equation*}
		as claimed.
	\end{proof}
	
	\subsection{Square functions}
	
	In this section, we record a series of lemmas that control the contribution of square functions at various scales. The proofs of these are standard, and have been delayed to the appendix.
	
	Our first lemma encodes that our frequency-localized functions $f_\theta$ and $f_{m,\theta}^\mathcal{B}$ are approximately constant on small scales.
	\begin{lemma}[Pointwise local constancy lemmas]\ 
		\begin{enumerate}
			\item[(a)] For any $\theta$, $|f_\theta|^2\lesssim|f_\theta|^2*|\rho_\theta^\vee|$.
			\item[(b)] For any $k,m$ and any $x$,
			\begin{equation*}
				|f_{m,\tau_k}|^2(x)\lesssim |f_{m,\tau_k}|^2*w_{R_k}(x).
			\end{equation*}
		\end{enumerate}
	\end{lemma}
	
	Our second lemma serves as a shorthand for passing between several integrals that are essentially equivalent to the wave-envelope expansion.
	
	\begin{lemma}[Integrated local constancy lemmas]\label{intlocalconstlemma}Let $r>0$ be dyadic.\ 
		\begin{enumerate}
			\item[(a)] If $r\lesssim (\log R)R_{k}/R$, then
			\begin{equation*}
				\int\Big|\sum_{\theta\subseteq\tau_k}|f_{m,\theta}^\mathcal{B}|^2*\eta_{\sim r}^\vee\Big|^2\lesssim \int\Big| \sum_{\theta\subseteq\tau_k}|f_{m,\theta}^\mathcal{B}|^2*|\rho_{C(\log R)U_{\tau_k,R}^*}^\vee|\,\Big|^2,
			\end{equation*}
            where $U_{\tau_k,R}^*$ is a rectangle of dimensions $R_k/R\times R^{-1}$ centered at the origin with long edge parallel to $\mathbf{t}_{c_{\tau_k}}$.
			\item[(b)] If $k\geq m$, then
			\begin{equation*}
				\int\Big|\sum_{\theta\subseteq\tau_k}|f_{m,\theta}^\mathcal{B}|^2*|\rho_{(\log R)U_{\tau_k,R}^*}^\vee|\,\Big|^2\lesssim(\log R)^2\sum_{U\in\mathcal{G}_{\tau_k}}|U|\left(\strokedint_U\sum_{\theta\subseteq\tau_k}|f_\theta|^2\right)^2+R^{-100}.
			\end{equation*}
		\end{enumerate}
		
	\end{lemma}
	
	Next, we note that, on the superlevel set $\big\{|f|>\alpha\big\}$, it is possible to replace $f$ by $f_N$, so we may appeal to the decomposition $f_N=f_1+\sum_{m=2}^Nf_m^\mathcal{B}$.
	
	\begin{lemma}[Replacement lemma]\label{replacement} $|f(x)-f_N(x)|\lesssim\frac{\alpha}{C_{\mathfrak{p}}^{1/2}(\log R)^8}$.
	\end{lemma}
	
	As a consequence, we will be able to control the size of the superlevel set $U_\alpha$ by the size of the auxiliary set $V_\alpha:=\{x:|f_N(x)|>\frac{1}{2}\alpha\}$.
	
	For the next lemma, we will need to define an adjacency relation.
	\begin{definition}\label{neardef}
		For caps $\tau_k,\tau_k'$ of the same size, we say ``$\tau_k\text{ near }\tau_{k'}$'' if $\text{dist}(\tau_k,\tau_{k'})\lesssim(\log R)\text{diam}(\tau_k)$ for a suitably chosen implicit constant. If $\tau_k,\tau_k'$ do not satisfy this, we write ``$\tau_k\text{ not near }\tau_k'$''.
	\end{definition}
	
	\begin{remark}
		As defined, we have that for each $\tau_k$, $\#\{\tau_k':\tau_k\text{ near }\tau_k'\}\lesssim\log R$.
	\end{remark}
	\begin{remark}\label{nearcomp}
		If $\tau_k\text{ near }\tau_k'$, then $\tau_k\subseteq C\log R(\tau_k'+(c_{\tau_k}-c_{\tau_k'}))$ and symmetrically.
	\end{remark}
	
	We now mention the two key lemmas that facilitate an efficient wave-envelope estimate. These are standard in high/low calculations, e.g. \cite{DGW} (in the proof of Theorem 5.4), \cite{GWZ} (in the proof of Lemma 1.4), \cite{FGM} (Lemmas 11, 12, 13), \cite{GGGHMW} (in the proof of Theorem 5), and \cite{GM1} (Lemmas 4, 5, 6). 
	
	\begin{lemma}[Low lemma]\label{lowlemma}
		For any $2\leq m\leq k\leq N$, $0\leq s\leq k$, and $r\leq (\log R)R_k^{-1}$,
		\begin{equation*}
			|f_{m,\tau_{s}}^{\mathcal{B}}|^2*\eta_{\leq r}^\vee(x)=\sum_{\tau_k\subseteq\tau_s}\sum_{\tau_k':\tau_k\text{ near }\tau_k'}\Big(f_{m,\tau_k}^{\mathcal{B}}\overline{f_{m,\tau_k'}^{\mathcal{B}}}\Big)*\eta_{\leq r}^\vee(x)
		\end{equation*}
		for any $x$ and any $\tau_s$.
	\end{lemma}
	
	\begin{lemma}[High Lemmas]\label{highlemma} For any $m,k,s$, and $\ell$ such that $2\leq m\leq N$, $0\leq s\leq k$, and $k+\ell\leq N$, and any cap $\tau_s$,
		\begin{enumerate}
			\item[(a)] 
			\begin{equation*}
				\int\Big|\sum_{\theta\subseteq\tau_s}|f_{m,\theta}^{\mathcal{B}}|^2*\eta_{\geq R_k/R}^\vee\Big|^2\lesssim\log R\sum_{\tau_k\subseteq\tau_s}\int\Big|\sum_{\theta\subseteq\tau_k}|f_{m,\theta}|^2*\eta_{\geq R_k/R}^\vee\Big|^2,
			\end{equation*}
			\item[(b)]
			\begin{equation*}
				\int\Big|\sum_{\tau_k\subseteq\tau_s}|f_{m,\tau_k}^{\mathcal{B}}|^2*\eta_{\geq R_k^{-1}}^\vee\Big|^2\lesssim(\log R)\sum_{\tau_k\subseteq\tau_s}\int|f_{m,\tau_k}^{\mathcal{B}}|^4,
			\end{equation*}
			\item[(c)]
			\begin{equation*}
				\int\Big|\sum_{\tau_{k}\subseteq\tau_s}\sum_{\tau_k'\,\mathrm{ near }\,\tau_k}(f_{m,\tau_{k}}^\mathcal{B}\overline{f_{m,\tau_{k}'}^\mathcal{B}})*\eta_{\geq R_{k+\ell}^{-1}}^\vee\Big|^2\lesssim(\log R)^{3}R_\ell\sum_{\tau_k\subseteq\tau_s}\int|f_{m,\tau_k}^\mathcal{B}|^4.
			\end{equation*}
		\end{enumerate}
	\end{lemma}
	
	Next, we will need a tool to ensure that, when taking wave envelope contributions of the bad parts $f_m^\mathcal{B}$, we are allowed to disregard the low-frequency envelopes which have not yet been pruned.
	
	\begin{lemma}[Weak high-domination of bad parts]\label{highdomlemma} Let $2\leq m\leq N$ and $0\leq k<m$.
		
		\begin{enumerate}
			\item[(a)] We have the estimate
			\begin{equation*}
				\big|\sum_{\tau_{m-1}\subseteq\tau_k}|f_{m,\tau_{m-1}}^{\mathcal{B}}|^2*\eta_{\leq R_{m-1}/R}^\vee(x)\big|\lesssim\frac{\alpha^2(\#\tau_{m-1}\subseteq\tau_k)}{C_{\mathfrak{p}}(\log R)^2(\#\tau_{m-1})^2}.
			\end{equation*}
			\item[(b)] Suppose $\alpha\lesssim(\log R)|f_{m,\tau_k}^\mathcal{B}(x)|$. Then
			\begin{equation*}
				\sum_{\tau_{m-1}\subseteq\tau_k}|f_{m,\tau_{m-1}}^{\mathcal{B}}|^2(x)\lesssim \big|\sum_{\tau_{m-1}\subseteq\tau_k}|f_{m,\tau_{m-1}}^{\mathcal{B}}|^2*\eta_{\geq R_{m-1}/R}^\vee(x)\big|.
			\end{equation*}
			
		\end{enumerate}
		
	\end{lemma}

	\section{Proof of Theorem \ref{squarefunction}}\label{sqfunctionproof}
	
	\subsection{Bounding the broad sets}
	
	This portion of the argument follows closely the approach of \cite{GM1}, Section 3. Recall that $U_\alpha$ is defined as the set
	\begin{equation*}
		U_\alpha=\big\{x\in B_R:|f(x)|>\alpha\big\}.
	\end{equation*}
	We consider also the auxiliary set
	\begin{equation*}
		V_\alpha=\Big\{x\in B_R:|f_N(x)|>\frac{1}{2}\alpha\Big\}.
	\end{equation*}
	To avoid trivialities, we assume $|U_\alpha|>0$ for the remainder of this section. By the replacement lemma \ref{replacement},
	\begin{equation*}
		U_\alpha\subseteq V_\alpha
	\end{equation*}
	for large enough $R$. By the pruning lemmas \ref{pruninglemmas},
	\begin{equation}\label{pruningchoice}
		\begin{split}
			V_\alpha\subseteq\big\{x\in V_\alpha:|f_1|(x)&\geq N^{-1}|f_N(x)|\big\}\cup\bigcup_{m=2}^N\big\{x\in V_\alpha:|f_m^{\mathcal{B}}|(x)\geq N^{-1}|f_N(x)|\big\}\\
			&=:U_\alpha^1\cup\bigcup_{m=2}^NU_\alpha^m
		\end{split}
	\end{equation}
	so that
	\begin{equation*}
		|U_\alpha|\leq|U_\alpha^1|+\sum_{m=2}^N|U_\alpha^m|.
	\end{equation*}
	We bound each of these sets in turn.
	
	\begin{proposition}[Case $m=1$]\label{broadc1}
		\[\alpha^4|U_\alpha^1|\lesssim C_{\mathfrak{p}}^2(\log R)^{19}\sum_{U\in\mathcal{G}_{\tau_1}}|U|\left(\strokedint_U\sum_{\theta}|f_\theta|^2\right)^2.\]
	\end{proposition}
	\begin{proof}
		Clearly it suffices to assume $|U_\alpha^1|>0$. Then there is some $x\in B_R$ such that $|f_1(x)|\geq \frac{1}{2N}\alpha$. Since
		\begin{equation*}
			\begin{split}
				\frac{1}{2N}\alpha\leq|f_1(x)|=&|\sum_{\tau_1}\sum_{\theta\subseteq\tau_1}\sum_{U\in\mathcal{G}_{\tau_1}}\psi_U(x)f_{2,\theta}(x)|\\
				&\leq|\sum_{\tau_1}\sum_{\theta\subseteq\tau_1}\sum_{U\in\mathcal{G}_{\tau_1};x\in C(\log R)^{2.5}U}\psi_U(x)f_{2,\theta}(x)|+|\sum_{\tau_1}\sum_{\theta\subseteq\tau_1}\sum_{U\in\mathcal{G}_{\tau_1};x\not\in C(\log R)^{2.5}U}\psi_U(x)f_{2,\theta}(x)|,
			\end{split}
		\end{equation*}
		and, if $x\not\in (\log R)^{2.5}U$, the near-exponential decay of $\psi_U$ implies
		\begin{equation*}
			|\psi_U(x)f_{2,\theta}(x)|\lesssim R^{-1000},
		\end{equation*}
		whereby
		\begin{equation*}
			|\sum_{\tau_1}\sum_{\theta\subseteq\tau_1}\sum_{U\in\mathcal{G}_{\tau_1};x\not\in C(\log R)^{2.5}U}\psi_U(x)f_{2,\theta}(x)|\leq R^{-100},
		\end{equation*}
		we conclude that there is some $\tau_1$ and $U\in\mathcal{G}_{\tau_1}$ with $x\in C(\log R)^{2.5}U$.
		
		Since $U\|U_{\tau_1,R}$, $U$ is a rectangle of dimensions $\frac{R}{O(1)}\times R$, and that by definition of $\mathcal{G}_{\tau_1}$ we have
		\begin{equation*}
			|U|^{-1}\int W_U\sum_{\theta\subseteq\tau_1}|f_{2,\theta}|^2\geq\frac{\alpha^2}{(\#\tau_1)^2}\frac{1}{C_{\mathfrak{p}}(\log R)^8}.
		\end{equation*}
		In particular,
		\begin{equation*}
			\alpha^4\leq C_{\mathfrak{p}}^2(\log R)^{16}\left(\strokedint_U\sum_{\theta}|f_\theta|^2\right)^2,
		\end{equation*}
		where we have used the pruning lemmas \ref{pruninglemmas}.
		
		The above calculation demonstrates that, for each $x\in U_\alpha$ satisfying $|f(x)|\leq 4N|f_1(x)|$, there is some $\tau_1$ and $U\in\mathcal{G}_{\tau_1}$ such that $x\in C(\log R)^{2.5}U$. Thus
		\begin{equation*}
			1_{\{x\in U_\alpha:|f(x)|\leq 2N|f_1(x)|\}}\leq\sum_{\tau_1}\sum_{U\in\mathcal{G}_{\tau_1}}1_{C(\log R)^{2.5}U},
		\end{equation*}
		and upon integrating we achieve
		\begin{equation*}
			\begin{split}
				|\big\{x\in U_\alpha:|f(x)|\leq 2N|f_1(x)|\big\}|&\leq \sum_{\tau_1}\sum_{U\in\mathcal{G}_{\tau_1}} (\log R)^5|U|\\
				&\leq 4\alpha^{-4}C_{\mathfrak{p}}^2(\log R)^{19}\sum_{\tau_1}\sum_{U\in\mathcal{G}_{\tau_1}}\left(\strokedint_U\sum_{\theta\subseteq\tau_1}|f_\theta|^2\right)^2,
			\end{split}
		\end{equation*}
		which rearranges to the desired
		\begin{equation*}
			\alpha^4|U_\alpha^1|\lesssim C_{\mathfrak{p}}^2(\log R)^{19}\sum_{\tau_1}\sum_{U\in\mathcal{G}_{\tau_1}}\left(\strokedint_U\sum_{\theta\subseteq\tau_1}|f_\theta|^2\right)^2.
		\end{equation*}
	\end{proof}
	
	We will use the following local bilinear restriction result, demonstrated in \cite{FGM}:
	\begin{theorem}[Bilinear restriction; Theorem 15 of \cite{FGM}]\label{bilrest} Let $S\geq 4$, $\frac{1}{2}\geq E\geq S^{-1/2}$, and $X\subseteq\R^2$ be Lebesgue measurable. Suppose that $\tau,\tau'$ are $E$-separated subsets of $\mathcal{N}_{S^{-1}}(\mathbb{P}^1)$. Then, for a partition $\Omega=\{\omega_S\}$ of $\mathcal{N}_{S^{-1}}(\mathbb{P}^1)$ into $\sim S^{-1/2}\times S^{-1}$-caps, we have
		\begin{equation*}
			\int_X|f_\tau|^2|f_{\tau'}|^2(x)dx\lesssim E^{-2}\int_{\mathcal{N}_{S^{1/2}}(X)}\big|\sum_{\omega_S}|f_{\omega_S}|^2*w_{S^{1/2}}(x)\big|^2dx.
		\end{equation*}
	\end{theorem}
	
	This will be our initial estimate when we try to estimate $f$ in the broad case. We now define the broad sets on which bilinear methods are appropriate.
	
	Define the $m$th $(2\leq m\leq N)$ broad sets in $U_\alpha$ to be as follows. Fix any $\tau_{k-1}$ and $\tau_k,\tau_k'\subseteq\tilde\tau_{k-1}$ non-adjacent caps, and define
	\begin{equation}\label{broadsetdef}
		\mathrm{Br}_\alpha^m(\tau_k,\tau_k')=\Big\{x\in U_\alpha^m:\alpha\lesssim (\log R)^2|f_{m,\tau_k}^\mathcal{B}f_{m,\tau_k'}^\mathcal{B}|^{1/2}\Big\}.
	\end{equation}

	\begin{proposition}[High domination of broad parts]\label{prop:hdom} For any such $\tau_k,\tau_k'$, and for $\ell=\max(m-1,k)$, we have
	    
\begin{equation*}
      \alpha^4|\mathrm{Br}_\alpha^m(\tau_k,\tau_k')|\lesssim(\log R)^{10}\int_{\R^2}\left|\sum_{\tau_\ell\subseteq\tau_{k-1}}|f_{m,\tau_\ell}^\cal B|^2*\eta_{\geq R_{\ell}/R}^\vee\right|^2.
  \end{equation*}
	\end{proposition}

	\begin{proof}
        By bilinear restriction,
        \begin{equation*}
            \int_{\mathrm{Br}_\alpha^m(\tau_k,\tau_k')}|f_{m,\tau_{k}}^\cal Bf_{m,\tau_k'}^\cal B|^2\lesssim(\log R)^{2}\int_{\cal{N}_{R_\ell}(\mathrm{Br}_\alpha^m(\tau_k,\tau_k'))}\left|\sum_{\tau_\ell\subseteq\tau_{k-1}}|f_{m,\tau_\ell}^\cal B|^2*w_{R_\ell}\right|^2.
        \end{equation*}
        By the weak high-domination lemma, for each $x\in\mathrm{Br}_\alpha^m(\tau_k,\tau_k')$,
        \begin{equation*}
            \sum_{\tau_\ell\subseteq\tau_{k-1}}|f_{m,\tau_\ell}^\cal B|^2(x)\lesssim\left|\sum_{\tau_\ell\subseteq\tau_{k-1}}|f_{m,\tau_\ell}^\cal B|^2*\eta_{\geq R_\ell/R}^\vee(x)\right|.
        \end{equation*}
        By the uncertainty principle and rapid decay of $w_{R_\ell}$, together with Young's inequality, we conclude that
        \begin{equation*}
            \int_{\cal{N}_{R_\ell}(\mathrm{Br}_\alpha^m(\tau_k,\tau_k'))}\left|\sum_{\tau_\ell\subseteq\tau_{k-1}}|f_{m,\tau_\ell}^\cal B|^2*w_{R_\ell}\right|^2\lesssim\int\left|\sum_{\tau_\ell\subseteq\tau_{k-1}}|f_{m,\tau_\ell}^\cal B|^2*\eta_{\geq R_\ell/R}^\vee\right|^2.
        \end{equation*}
    
\end{proof}

\begin{proposition}[Case $2\leq m\leq N$]\label{broadc2} Let $1\leq k\leq m\leq N$. Suppose $\tau_k,\tau_k'\subseteq\tilde\tau_{k-1}$ are non-adjacent. Then, for $\ell=\max(m-1,k)$, we have
\begin{equation*}
      \int_{\R^2}\left|\sum_{\tau_\ell\subseteq\tau_{k-1}}|f_{m,\tau_\ell}^\cal B|^2*\eta_{\geq R_{\ell}/R}^\vee\right|^2\lesssim(\log R)^{9}\sum_{\ell\leq \nu\leq N}\sum_{\tau_\nu\subseteq\tau_{k-1}}\sum_{U\in\mathcal{G}_{\tau_k}}|U|\left(\strokedint_U\sum_{\theta\subseteq\tau_\nu}|f_\theta|^2\right)^2+R^{-50}.
  \end{equation*}

\end{proposition}
  \begin{proof}
		We pigeonhole to a dyadic scale. Let $R_{\ell}/R\leq r\leq 2N R_{\ell}^{-1}$ be dyadic such that
		\begin{equation*}
			\int_{\R^2}\Big|\sum_{\tau_{\ell}\subseteq\tau_{k-1}}|f_{m,\tau_{\ell}}^\mathcal{B}|^2*\eta_{\geq R_{\ell-1}/R}^\vee\Big|^2\lesssim\log R\int_{\R^2}\Big|\sum_{\tau_{\ell}\subseteq\tau_{k-1}}|f_{m,\tau_{\ell}}^\mathcal{B}|^2*\eta_{\geq R_{\ell-1}/R}^\vee*\eta_{\sim r}^\vee\Big|^2.
		\end{equation*}
		By Young,
		\begin{equation*}
			\int_{\R^2}\Big|\sum_{\tau_{\ell}\subseteq\tau_{k-1}}|f_{m,\tau_{\ell}}^\mathcal{B}|^2*\eta_{\geq R_{\ell-1}/R}^\vee*\eta_{\sim r}^\vee\Big|^2\lesssim\int_{\R^2}\Big|\sum_{\tau_{\ell}\subseteq\tau_{k-1}}|f_{m,\tau_{\ell}}^\mathcal{B}|^2*\eta_{\sim r}^\vee\Big|^2.
		\end{equation*}
		The remainder of the analysis will be split into cases, depending on the size of $r$.
		
		\underline{Case 1: $r\leq R^{-1/2}$}. By the low lemma \ref{lowlemma},
		\begin{equation*}
			\int\Big|\sum_{\tau_{\ell}\subseteq\tau_{k-1}}|f_{m,\tau_{\ell}}^\mathcal{B}|^2*\eta_{\sim r}^\vee\Big|^2=\int\Big|\sum_{\theta\subseteq\tau_{k-1}}\sum_{\theta'\subseteq\tau_{k-1}\text{ near }\theta}|f_{m,\theta}^\mathcal{B}|^2*\eta_{\sim r}^\vee\Big|^2.
		\end{equation*}
		Let $k$ be s.t. $r\sim R_{k-1}/R$. Since we have assumed $r\geq R_{m-1}/R$, we must have $k\geq m$. By the triangle inequality and the wave envelope expansion lemma \ref{intlocalconstlemma}, we have
		\begin{equation*}
			\int\Big|\sum_{\theta\subseteq\tau_{k-1}}\sum_{\theta'\subseteq\tau_{k-1}\text{ near }\theta}|f_{m,\theta}^\mathcal{B}|^2*\eta_{\sim r}^\vee\Big|^2\lesssim(\log R)^2\sum_{\tau_k\subseteq\tau_{k-1}}\sum_{U\in\mathcal{G}_{\tau_k}}|U|\left(\strokedint_U\sum_{\theta\subseteq\tau_k}|f_\theta|^2\right)^2.
		\end{equation*}
		We conclude that
		\begin{equation*}
			\int_{\R^2}\Big|\sum_{\tau_{\ell}\subseteq\tau_{k-1}}|f_{m,\tau_{\ell}}^\mathcal{B}|^2*\eta_{\geq R_{\ell-1}/R}^\vee\Big|^2\lesssim (\log R)^{3}\sum_{\tau_k\subseteq\tau_{k-1}}\sum_{U\in\mathcal{G}_{\tau_k}}|U|^{-1}\left(\int\sum_{\theta\subseteq\tau_k}|f_\theta|^2W_U\right)^2
		\end{equation*}
		as claimed.
		
		\underline{Case 2: $r>R^{-1/2}$} Let $s=N$ if $r\leq (\log R)R^{-1/2}$, and otherwise choose $\ell\leq s<N$ such that $(\log R)R_{s+1}^{-1}<r\leq (\log R)R_s^{-1}$. By the low lemma \ref{lowlemma},
		\begin{equation*}
			\int\Big|\sum_{\tau_\ell\subseteq\tau_{k-1}}|f_{m,\tau_{\ell}}^\mathcal{B}|^2*\eta_{\sim r}^\vee\Big|^2=\int\Big|\sum_{\tau_{s}\subseteq\tau_{k-1}}\sum_{\tau_{s}'\subseteq\tau_{k-1}\text{ near }\tau_{s}}(f_{m,\tau_{s}}^\mathcal{B}\overline{f_{m,\tau_{s}'}^\mathcal{B}})*\eta_{\sim r}^\vee\Big|^2.
		\end{equation*}
		By part (c) of the high lemma \ref{highlemma},
		\begin{equation*}
				\int\Big|\sum_{\tau_{s}\subseteq\tau_{k-1}}\sum_{\tau_{s}'\subseteq\tau_{k-1}\text{ near }\tau_{s}}(f_{m,\tau_{s}}^\mathcal{B}\overline{f_{m,\tau_{s}'}^\mathcal{B}})*\eta_{\sim r}^\vee\Big|^2\lesssim(\log R)^3\sum_{\tau_{s}}\int|f_{m,\tau_{s}}^\mathcal{B}|^4.
		\end{equation*}
		By the reverse square function estimate for $\mathbb{P}^1$ and by splitting $f_{m,\tau_s}^{\mathcal{B}}$ into $O(\log R)$ pieces with disjoint Fourier support,
		\begin{equation*}
			\int|f_{m,\tau_s}^\mathcal{B}|^4\lesssim (\log R)^4\int|\sum_{\theta\subseteq\tau_s}\left|f_{m,\theta}^{\mathcal{B}}|^2\right|^2.
		\end{equation*}
		
		So far, in case 2, we have reached the estimate
		\begin{equation*}
			\int_{\R^2}\Big|\sum_{\tau_{\ell}\subseteq\tau_{k-1}}|f_{m,\tau_{\ell}}^\mathcal{B}|^2*\eta_{\geq R_{\ell-1}/R}^\vee\Big|^2\lesssim(\log R)^{8}\sum_{\tau_s\subseteq\tau_{k-1}}\int\Big|\sum_{\theta\subseteq\tau_s}|f_{m,\theta}^\mathcal{B}|^2\Big|^2
		\end{equation*}
		for some $s\geq m$. We consider two sub-cases, depending on if the latter is high- or low-dominated.
		
		\underline{Case 2a}: Suppose that
		\begin{equation*}
			\sum_{\tau_s\subseteq\tau_{k-1}}\int\Big|\sum_{\theta\subseteq\tau_s}|f_{m,\theta}^\mathcal{B}|^2\Big|^2\lesssim\sum_{\tau_s\subseteq\tau_{k-1}}\int\Big|\sum_{\theta\subseteq\tau_s}|f_{m,\theta}^{\mathcal{B}}|^2*\eta_{\leq R_m/R}^\vee\Big|^2.
		\end{equation*}
		Since $m\leq s$, we have by the wave envelope expansion lemma \ref{intlocalconstlemma}
		\begin{equation*}
			\sum_{\tau_s\subseteq\tau_{k-1}}\int\Big|\sum_{\theta\subseteq\tau_s}|f_{m,\theta}^{\mathcal{B}}|^2*\eta_{\leq R_m/R}^\vee\Big|^2\lesssim\sum_{\tau_m\subseteq\tau_{k-1}}\sum_{U\in\mathcal{G}_{\tau_m}}|U|\left(\strokedint_U\sum_{\theta\subseteq\tau_m}|f_{\theta}|^2\right)^2+R^{-100}.
		\end{equation*}
		Thus we have the desired
		\begin{equation*}
			\sum_{\tau_s\subseteq\tau_{k-1}}\int\Big|\sum_{\theta\subseteq\tau_s}|f_{m,\theta}^{\mathcal{B}}|^2\Big|^2\lesssim\sum_{\tau_m\subseteq\tau_{k-1}}\sum_{U\in\mathcal{G}_{\tau_m}}|U|\left(\strokedint_U\sum_{\theta\subseteq\tau_m}|f_\theta|^2\right)^2+R^{-100}.
		\end{equation*}
		
		\underline{Case 2b}: If we are not in case 2a, then
		\begin{equation*}
			\sum_{\tau_s\subseteq\tau_{k-1}}\int\Big|\sum_{\theta\subseteq\tau_s}|f_{m,\theta}^{\mathcal{B}}|^2\,\Big|^2\lesssim\sum_{\tau_s\subseteq\tau_{k-1}}\int\Big|\sum_{\theta\subseteq\tau_s}|f_{m,\theta}^{\mathcal{B}}|^2*\eta_{\geq R_m/R}^\vee\Big|^2.
		\end{equation*}
		Now let $\mu$ be dyadic between $R_m/R$ and $(\log R)R^{-1/2}$ such that
		\begin{equation*}
			\sum_{\tau_s\subseteq\tau_{k-1}}\int\Big|\sum_{\theta\subseteq\tau_s}|f_{m,\theta}^{\mathcal{B}}|^2*\eta_{\geq R_m/R}^\vee\Big|^2\lesssim\log R\sum_{\tau_s\subseteq\tau_{k-1}}\int\Big|\sum_{\theta\subseteq\tau_s}|f_{m,\theta}^{\mathcal{B}}|^2*\eta_{\sim\mu}^\vee\Big|^2.
		\end{equation*}
		If $\mu\leq R_s/R$, then by the integrated local constancy \ref{intlocalconstlemma} we have
		\begin{equation*}
			\sum_{\tau_s\subseteq\tau_{k-1}}\int\Big|\sum_{\theta\subseteq\tau_s}|f_{m,\theta}^{\mathcal{B}}|^2*\eta_{\sim \mu}^\vee\Big|^2\lesssim\sum_{\tau_s\subseteq\tau_{k-1}}\sum_{U\in\mathcal{G}_{\tau_s}}|U|\left(\strokedint_U\sum_{\theta\subseteq\tau_k}|f_\theta|^2\right)^2 + R^{-100},
		\end{equation*}
		and we are done.
		
		On the other hand, if $\mu>R_s/R$, then pick $p\geq s$ such that $R_p/R\leq\mu<R_{p+1}/R$. Then by the high lemma \ref{highlemma}
		\begin{equation*}
			\sum_{\tau_s\subseteq\tau_{k-1}}\int\Big|\sum_{\theta\subseteq\tau_s}|f_{m,\theta}^{\mathcal{B}}|^2*\eta_{\sim \mu}^\vee\Big|^2\lesssim(\log R)\sum_{\tau_p\subseteq\tau_{k-1}}\int\Big|\sum_{\theta\subseteq\tau_p}|f_{m,\theta}^{\mathcal{B}}|^2*\eta_{\sim \mu}^\vee\Big|^2,
		\end{equation*}
		and as above, by the wave envelope expansion lemma,
		\begin{equation*}
			\begin{split}
				\sum_{\tau_p\subseteq\tau_{k-1}}\int\Big|\sum_{\theta\subseteq\tau_p}|f_{m,\theta}^{\mathcal{B}}|^2*\eta_{\sim \mu}^\vee\Big|^2&\lesssim\sum_{\tau_p\subseteq\tau_{k-1}}\int\Big|\sum_{\theta\subseteq\tau_p}|f_{m,\theta}^{\mathcal{B}}|^2*|\rho_{\tau_p}^\vee|\Big|^2\\
				&\lesssim\sum_{\tau_p\subseteq\tau_{k-1}}\sum_{U\in\mathcal{G}_{\tau_p}}|U|\left(\strokedint_U\sum_{\theta\subseteq\tau_p}|f_\theta|^2\right)^2+R^{-100},
			\end{split}
		\end{equation*}
		from which we have the estimate
		\begin{equation*}
			\int_{\R^2}\Big|\sum_{\tau_{\ell}\subseteq\tau_{k-1}}|f_{m,\tau_{\ell}}^\mathcal{B}|^2*\eta_{\geq R_{\ell-1}/R}^\vee\Big|^2\lesssim(\log R)^{9}\sum_{\ell\leq\nu\leq N}\sum_{\tau_\nu\subseteq\tau_{k-1}}\sum_{U\in\mathcal{G}_{\tau_\nu}}|U|\left(\strokedint_U\sum_{\theta\subseteq\tau_\nu}|f_\theta|^2\right)^2 + R^{-50},
		\end{equation*}
		and we are done.
		
	\end{proof}
	
	\subsection{Broad/narrow analysis}
	
	In Propositions \ref{broadc1} and \ref{broadc2}, we produced the desired bounds on the subset of the superlevel set for which $f$ is sufficiently broad at some scale. In this subsection, we perform a broad/narrow analysis to produced the desired wave envelope estimate in each cube of sidelength $R$.
	
	As a note: for the remainder of the article, we suppress the constant $C_{\mathfrak{p}}$ from the pruning definition as an implicit constant.
	
	\begin{proposition}[Local wave envelope estimate]\label{broadnarrow} For each cube $B_R$ of sidelength $R$ and each $\alpha>0$,
		\begin{equation*}
			\alpha^4|\{x\in B_R:|f(x)|>\alpha\}|\lesssim(\log R)^{20}\sum_{\substack{R^{-1/2}\leq s\leq 1\\s\text{ dyadic}}}\sum_{\tau:\ell(\tau)=s}\sum_{U\in\mathcal{G}_\tau}|U|^{-1}\|S_Uf\|_2^4.
		\end{equation*}
		
	\end{proposition}
	
	We first note a technical obstruction. The common strategy in decoupling theory for performing broad/narrow analysis can be summarized as follows. Fix some scale $s$ and $x\in B_R$, and fix $\tau_*$ to be the box of size $d(\tau_*)=s$ which maximizes $|f_{\tau_*}(x)|$. Then since $f(x)=\sum_\tau f_\tau(x)$, it follows (Lemma 7.2 of \cite{Dbook}) that either
	\begin{equation*}
		|f(x)|\leq 4|f_{\tau_*}(x)|\quad\text{or}\quad |f(x)|\leq s^{-\frac{3}{2}}\max_{\tau\text{ not near }\tau'}|f_\tau(x)f_{\tau'}(x)|^{1/2},
	\end{equation*}
	where the maximum is taken over those boxes $\tau,\tau'$ of diameter $d(\tau)=d(\tau')=s$. If we simplify the above as
	\begin{equation*}
		|f(x)|\leq 4|f_{\tau_*}(x)|+ s^{-\frac{3}{2}}\max_{\tau\text{ not near }\tau'}|f_\tau(x)f_{\tau'}(x)|^{1/2}
	\end{equation*}
	and iterate by first choosing $s=R_1^{-1}$, then breaking up the first summand by choosing $s=R_2^{-1}$ and rescaling, etc., we achieve the estimate
	\begin{equation*}
		|f(x)|\leq 4^N\max_\theta|f_\theta(x)|+P(x),
	\end{equation*}
	for a suitable nonnegative quantity $P(x)$. Note however that $4^N=R^{O(1)}$, which is much too large. In the classical sequence of scales $R_k=R^{k\eps}$ or $(\log R)^k$, this broad/narrow analysis would still be larger than our desired error $(\log R)^{O(1)}$ (while nevertheless being  $O_\eps(R^\eps)$).
	
	As a consequence, the broad/narrow analysis will need to be carried out more efficiently. We follow an approach demonstrated in Section 4 of \cite{GMW}, where a $(\log R)^{O(1)}$ error was obtained for canonical-scale ($\beta=\frac{1}{2}$) decoupling. Namely, the domain of integration for $|f|^4$ will be successively divided into broad and narrow sets, ranging over many scales. If a point $x$ is broad at some scale, we will be able to productively use Propositions \ref{broadc1} and \ref{broadc2}. If instead $x$ is narrow at all scales, then a trivial estimate will suffice. As suggested by the above analysis, we will need to reduce the factor $4$ to a quantity that does not grow too quickly under the iteration.
	
	We proceed to the proof. We will express the various estimates as ``decoupling'' bounds, though it is worth emphasizing that they are arranged \textit{pointwise} (so this decomposition scheme is really a decomposition of \textit{constants}, not functions with special Fourier support); we do so because of the convenient inductive structure of decoupling-style bounds.
 
    Fix $2\leq m\leq N$. We first present a modification of Lemma 8 of \cite{GMW}, which serves to replace the constant $4$ in the prior calculation with a much smaller quantity.
	\begin{lemma}[Narrow lemma]\label{narrow} For all sufficiently large $R$, the following holds. Suppose $1\leq k\leq N$ and $\tau_{k-1}$ is a cap of diameter $R_{k-1}^{-1}$. Let $\{\tau_{k}\}$ be the caps of diameter $R_{k}^{-1}$ with $\tau_{k}\subseteq\tilde\tau_{k-1}$. Then, for each $x$, either
		\begin{equation}\label{broadest}
			|f_{m,\tilde\tau_{k-1}}^\mathcal{B}(x)|\lesssim(\log R)\max_{\tau_{k}\mathrm{\ not\ adj.\ to\ }\tau_{k}'}|f_{m,\tau_{k}}^\mathcal{B}(x)f_{m,\tau_{k}'}^\mathcal{B}(x)|^{1/2}
		\end{equation}
		or
		\begin{equation}\label{narrowest}
			|f_{m,\tilde\tau_{k-1}}^\mathcal{B}(x)|\leq\left(1+\frac{1}{\log R}\right)\max_{\tau_{k}\subseteq\tilde\tau_{k-1}}|f_{m,\tilde\tau_{k}}^\mathcal{B}(x)|.
		\end{equation}
		
	\end{lemma}
	\begin{proof}
		Fix $\tau_{k}^*\subseteq\tau_{k-1}$ which realizes the maximum
		\begin{equation*}
			|f_{m,\tilde\tau_{k}^*}^\mathcal{B}(x)|=\max_{\tau_{k}\subseteq\tilde\tau_{k-1}}|f_{m,\tilde\tau_{k}}^\mathcal{B}(x)|.
		\end{equation*}
		Suppose \ref{narrowest} fails. Then, since $f_{m,\tilde\tau_{k-1}}^\mathcal{B}(x)=\sum_{\tau_{k}\subseteq\tilde\tau_{k-1}}f_{m,\tau_{k}}^\mathcal{B}(x)$, we have the inequality
		\begin{equation*}
			|f_{m,\tilde\tau_{k-1}}^\mathcal{B}(x)-\sum_{\tau_{k}\mathrm{\ not\ adj.\ to\  }\tau_{k}^*}f_{m,\tau_{k}}^\mathcal{B}(x)|<\left(1+\frac{1}{\log R}\right)^{-1}|f_{m,\tilde\tau_{k-1}}^\mathcal{B}(x)|.
		\end{equation*}
		On the other hand,
		\begin{equation*}
			|f_{m,\tilde\tau_{k-1}}^\mathcal{B}(x)-\sum_{\tau_{k}\mathrm{\ not\ adj.\ to\  }\tau_{k}^*}f_{m,\tau_{k}}^\mathcal{B}(x)|\geq|f_{m,\tilde\tau_{k-1}}^\mathcal{B}(x)|-(\#\tau_{k}\subseteq\tilde\tau_{k-1})\max_{\tau_{k}\mathrm{\ not\ adj.\ to\  }\tau_{k}^*}|f_{m,\tau_{k}}^\mathcal{B}(x)|;
		\end{equation*}
		the above implies
		\begin{equation*}
			(\#\tau_{k+1}\subseteq\tau_k)\max_{\tau_{k}\mathrm{\ not\ adj.\ to\  }\tau_{k}^*}|f_{m,\tau_{k}}^\mathcal{B}(x)|>\left(1-\left(1+\frac{1}{\log R}\right)^{-1}\right)|f_{m,\tilde\tau_{k-1}}^\mathcal{B}(x)|.
		\end{equation*}
		Relating the above to \ref{broadest}, for each $\tau_{k}\subseteq\tilde\tau_{k-1}$,
		\begin{equation*}
			|f_{m,\tau_{k}}^\mathcal{B}(x)|\leq|f_{m,\tau_{k}}^\mathcal{B}(x)f_{m,\tau_{k}^*}^\mathcal{B}(x)|^{1/2},
		\end{equation*}
		and thus
		\begin{equation*}
			|f_{m,\tilde\tau_{k-1}}^\mathcal{B}(x)|<(\#\tau_{k}\subseteq\tilde\tau_{k-1})\left(1-\left(1+\frac{1}{\log R}\right)^{-1}\right)^{-1}\max_{\tau_{k}\mathrm{\ not\ adj.\ to\  }\tau_{k}^*}|f_{m,\tau_{k}}^\mathcal{B}(x)f_{m,\tau_{k}^*}^\mathcal{B}(x)|^{1/2}.
		\end{equation*}
		The conclusion follows from the estimates
		\begin{equation*}
			\left(1-\left(1+\frac{1}{\log R}\right)^{-1}\right)^{-1}\lesssim\log R
		\end{equation*}
		and $\#(\tau_{k}\subseteq\tilde\tau_{k-1})\lesssim 1$.
	\end{proof}
	
	We wish to use this to divide the integral of $|f_m^\mathcal{B}|^4$ into broad and narrow parts, with a small constant on narrow parts. For the narrow component, we wish to relate $\int |f_m^\mathcal{B}|^4$ to $\sum_\tau\int|f_{m,\tau}^\mathcal{B}|^4$, so that we may further decompose each $f_{m,\tau}^\mathcal{B}$ into broad and narrow components and proceed inductively.
	
	\begin{definition} We define $\text{Broad}_{1,m}$ to be the set
		\begin{equation*}
			\text{Broad}_{1,m}=\left\{x\in U_\alpha^m:\,\alpha\lesssim (\log R)^2\max_{\tau_1\mathrm{\ not\ adj.\ to\  }\tau_1'}|f_{m,\tau_1}^\mathcal{B}(x)f_{m,\tau_1'}^\mathcal{B}(x)|^{1/2}\right\}.
		\end{equation*}
		The complementary set $\mathrm{Narrow}_{1,m}$ is defined as $U_\alpha^m\setminus\mathrm{Broad}_{1,m}$.
	\end{definition}

	\begin{remark}
		It follows that $\mathrm{Broad}_{1,m}$ may be covered by $O((\log R)^2)$-many $\mathrm{Br}_\alpha^m(\tau,\tau')$.
	\end{remark}

	\begin{definition}
		Write, for each $\tau_1$,
		\begin{equation*}
            \begin{split}
			&\text{Broad}_{2,m}(\tau_1):=\\
        &\left\{x\in\text{Narrow}_{1,m}:|f_{m,\tilde\tau_1}^\mathcal{B}(x)|\lesssim\left(1+\frac{1}{\log R}\right)(\log R)^2\max_{\substack{\tau_2\mathrm{\ not\ adj.\ to\  }\tau_2'\\\tau_2,\tau_2'\subseteq\tilde\tau_1}}|f_{m,\tau_2}^\mathcal{B}(x)f_{m,\tau_2'}^\mathcal{B}(x)|^{1/2}\right\}
            \end{split}
		\end{equation*}
		where as usual each $\tau_2$ has diameter $\sim R_2^{-1}$. Write also $\text{Narrow}_{2,m}(\tau_1):=\text{Narrow}_{1,m}\setminus\text{Broad}_{2,m}(\tau_1)$.
	\end{definition}
	
	\begin{definition}
		Let $2\leq k<N$. Suppose $\tau_k\subseteq\tau_{k-1}$ have diameter $\sim R_k^{-1},\sim R_{k-1}^{-1}$, respectively. We inductively write
		\begin{equation*}
			\begin{split}
				&\mathrm{Broad}_{k+1,m}(\tau_k):=\\
        &\left\{x\in\mathrm{Narrow}_{k,m}(\tau_{k-1}):\alpha\lesssim\left(1+\frac{1}{\log R}\right)^{k}(\log R)^2\max_{\substack{\tau_{k+1}\mathrm{\ not\ adj.\ to\  }\tau_{k+1}'\\\tau_{k+1},\tau_{k+1}'\subseteq\tilde\tau_{k}}}|f_{m,\tau_{k+1}}^\mathcal{B}(x)f_{m,\tau_{k+1}'}^\mathcal{B}(x)|^{1/2}\right\}
			\end{split}
		\end{equation*}
		and $\text{Narrow}_{k+1,m}(\tau_k):=\text{Narrow}_{k,m}(\tau_{k-1})\setminus\text{Broad}_{k+1,m}(\tau_k)$.
	\end{definition}
	
	It follows directly from the definitions of these sets that
	\begin{equation}\label{decompofum}
		\begin{split}
			\alpha^4|U_\alpha^m|&=\alpha^4|\mathrm{Broad}_{1,m}|+\alpha^4|\mathrm{Narrow}_{1,m}|\\
                &\leq\alpha^4|\mathrm{Broad}_{1,m}|+\alpha^4\sum_{\tau_1}|\mathrm{Narrow}_{2,m}(\tau_1)|+|\mathrm{Broad}_{2,m}(\tau_1)|\\
                &\quad\quad\vdots\\
                &\leq\alpha^4|\mathrm{Broad}_{1,m}|+\alpha^4\sum_{k=2}^N\sum_{\tau_k}|\mathrm{Broad}_{k+1,m}(\tau_k)|+\alpha^4\sum_{\tau_{N-1}}|\mathrm{Narrow}_{N,m}(\tau_{N-1})|.
		\end{split}
	\end{equation}
	
	We bound each of the preceding summands in turn.
	
	\begin{lemma}[Broad bound, $k=1$]\label{broadbound1p2} We have
		\begin{equation*}
			\alpha^4|\mathrm{Broad}_{1,m}|\lesssim(\log R)^{20}\sum_{m<k<N}\sum_{\tau:\ell(\tau)=R_k^{-1}}\sum_{U\in\mathcal{G}_\tau}|U|\left(\strokedint_U\sum_{\theta\subseteq\tau}|f_\theta|^2\right)^2.
		\end{equation*}
	\end{lemma}
	\begin{proof}
		
		Suppose first $m=1$. Then $\mathrm{Broad}_{1,1}\subseteq U_\alpha^1$, so by \ref{broadc1} we have
		\begin{equation*}
			\alpha^4|\mathrm{Broad}_{1,1}|\lesssim (\log R)^{19}\sum_{U\in\mathcal{G}_{\tau_1}}|U|\left(\strokedint_U\sum_{\theta}|f_\theta|^2\right)^2.
		\end{equation*}
		
		Suppose next $2\leq m\leq N$. By the definition of the first broad set,
		\begin{equation*}
			\mathrm{Broad}_{1,m}=\bigcup_{\tau_1\text{ not near }\tau_1'}\mathrm{Br}_\alpha^m(\tau_1,\tau_1'),
		\end{equation*}
		and so, since there are $O(1)$-many $\tau_1$,
		\begin{equation*}
			\alpha^4|\mathrm{Broad}_{1,m}|\lesssim(\log R)^2\alpha^4\max_{\tau_1\text{ not near }\tau_1'}|\mathrm{Br}_\alpha^m(\tau_1,\tau_1')|.
		\end{equation*}
		By Prop.'s \ref{prop:hdom} and \ref{broadc2}, we conclude that
		\begin{equation*}
			\alpha^4|\mathrm{Broad}_{1,m}|\lesssim(\log R)^{19}\sum_{m\leq k\leq N}\sum_{\tau_k}\sum_{U\in\mathcal{G}_{\tau_k}}|U|\left(\strokedint_U\sum_{\theta\subseteq\tau}|f_\theta|^2\right)^2.
		\end{equation*}
		
	\end{proof}

 \begin{lemma}[Broad bound, $2\leq k\leq N$]\label{broadbound2p2} We have
		\begin{equation*}
			\alpha^4\sum_{\tau_{k-1}}|\mathrm{Broad}_{k,m}(\tau_{k-1})|\lesssim(\log R)^{19}\sum_{m\leq s\leq N}\sum_{\tau_s}\sum_{U\in\mathcal{G}_\tau}|U|\left(\strokedint_U\sum_{\theta\subseteq\tau_s}|f_{m,\theta}^\mathcal{B}|^2\right)^2.
		\end{equation*}
	\end{lemma}
	\begin{proof}
		By the definition of the broad set, for each $\tau_{k-1}$ and each $x\in\mathrm{Broad}_{k,m}(\tau_{k-1})$ there is some pair $\tau_k,\tau_k'\subseteq\tilde\tau_{k-1}$ non-adjacent such $\alpha\lesssim(1+\frac{1}{\log R})^k(\log R)^2|f_{m,\tau_k}^\mathcal{B}(x)f_{m,\tau_k'}^\mathcal{B}(x)|^{1/2}$; since $(1+\frac{1}{\log R})^N\lesssim 1$, we find that $x\in\mathrm{Br}_\alpha^m(\tau_k,\tau_k')$. In other words,
		\begin{equation*}
			\mathrm{Broad}_{k,m}(\tau_{k-1})\subseteq\bigcup_{\substack{\tau_k,\tau_k'\subseteq\tilde\tau_{k-1}\\\tau_k\mathrm{\ not\ adj.\ to\  }\tau_k'}}\mathrm{Br}_\alpha^m(\tau_k,\tau_k').
		\end{equation*}

        By Prop.'s \ref{prop:hdom} and \ref{broadc2},
        \begin{equation*}
            \alpha^4|\mathrm{Br}_\alpha^m(\tau_k,\tau_k')|\lesssim(\log R)^{19}\sum_{m-1\leq s\leq N}\sum_{\tau_s\subseteq\tilde\tau_{k-1}}\sum_{U\in\mathcal{G}_\tau}|U|\left(\strokedint_U\sum_{\theta\subseteq\tau_s}|f_{m,\theta}^\mathcal{B}|^2\right)^2;
        \end{equation*}
        the result follows immediately.
		
	\end{proof}

    \begin{lemma}[Narrow bound]\label{narrowboundp2} We have
		\begin{equation*}
			\alpha^4\sum_{\tau_{N-1}}|\mathrm{Narrow}_{N,m}(\tau_{N-1})|\lesssim(\log R)^4\sum_{U\in\mathcal{G}_\theta}|U|^{-1}\left(\int|f_\theta|^2(y)W_U(y)dy\right)^2.
		\end{equation*}
	\end{lemma}
	\begin{proof}
		Note that each $\tilde\theta$ is equal to a union of $\leq 3$ distinct $\theta$. In particular, for each $x\in\mathrm{Narrow}_{N,m}(\tau_{N-1})$,
		\begin{equation*}
			\sum_{\theta\subseteq\tilde\tau_{N-1}}|f_{m,\tilde\theta}^\mathcal{B}(x)|^4\leq 3^3\sum_{\theta\subseteq\tilde\tau_{N-1}}|f_{m,\theta}^\mathcal{B}(x)|^4,
		\end{equation*}
		hence
		\begin{equation*}
			\begin{split}
				\sum_{\tau_{N-1}}\int_{\mathrm{Narrow}_{N,m}(\tau_{N-1})}\sum_{\theta\subseteq\tilde\tau_{N-1}}|f_{m,\tilde\theta}^\mathcal{B}|^4&\lesssim\sum_{\tau_{N-1}}\int_{\mathrm{Narrow}_{N,m}(\tau_{N-1}^*)}\sum_{\theta\subseteq\tilde\tau_{N-1}}|f_{m,\theta}^\mathcal{B}|^4\\
				&\leq\sum_{\tau_{N-1}}\int_{B_R}\sum_{\theta\subseteq\tilde\tau_{N-1}}|f_{N,\theta}|^4.
			\end{split}
		\end{equation*}
		By the definition of the pruning, for each $\theta$,
		\begin{equation*}
			\int_{B_R}|f_{N,\theta}|^4=\int_{B_R}|\sum_{U\in\mathcal{G}_\theta}\psi_Uf_\theta|^4\lesssim\sum_{U\in\mathcal{G}_\theta}\int_{B_R}|\psi_U|^4|f_\theta|^4.
		\end{equation*}
		Since each $|\psi_U|\leq 1$, we have the trivial bound
		\begin{equation*}
			\sum_{U\in\mathcal{G}_\theta}\int_{B_R}|\psi_U|^4|f_\theta|^4\leq\sum_{U\in\mathcal{G}_\theta}\int_{B_R}|\psi_U|^2|f_\theta|^4.
		\end{equation*}
		By the local constancy lemma (a),
		\begin{equation*}
			\begin{split}
				\sum_{U\in\mathcal{G}_\theta}\int_{B_R}|\psi_U|^2|f_\theta|^4&\lesssim\sum_{U\in\mathcal{G}_\theta}\int_{B_R}|\psi_U|^2\left(|f_\theta|^2*|\rho_\theta^\vee|\right)^2\\
				&=\sum_{U\in\mathcal{G}_\theta}\int_{B_R}|\psi_U|^2(x)\left(\int|f_\theta|^2(y)|\rho_\theta^\vee|(x-y)dy\right)^2dx.
			\end{split}
		\end{equation*}
		By Minkowski,
		\begin{equation*}
			\begin{split}
				\sum_{U\in\mathcal{G}_\theta}\int_{B_R}|\psi_U|^2(x)&\left(\int|f_\theta|^2(y)|\rho_\theta^\vee|(x-y)dy\right)^2dx\\
				&\leq\sum_{U\in\mathcal{G}_\theta}\left(\int|f_\theta|^2(y)\left(\int|\psi_U|^2(x)|\rho_\theta^\vee|^2(x-y)dx\right)^{1/2}dy\right)^2\\
				&\leq\sum_{U\in\mathcal{G}_\theta}\left(\int|f_\theta|^2(y)\left(\int|\psi_U|^2(x)|\rho_\theta^\vee|^2(x-y)dx\right)^{1/2}dy\right)^2.
			\end{split}
		\end{equation*}
		By the rapid decay of $\rho_\theta^\vee$ outside of $\theta^*$,
		\begin{equation*}
			\int|\psi_U|^2(x)|\rho_\theta^\vee|^2(x-y)dx\lesssim\sup_{x\in y+\theta^*}|\psi_U|^2(x)\int|\rho_\theta^\vee|^2(x-y)dx\lesssim W_U^2(y)|U|^{-1},
		\end{equation*}
		and so
		\begin{equation*}
			\sum_{U\in\mathcal{G}_\theta}\left(\int|f_\theta|^2(y)\left(\int|\psi_U|^2(x)|\rho_\theta^\vee|^2(x-y)dx\right)^{1/2}\right)^2dy\lesssim\sum_{U\in\mathcal{G}_\theta}|U|^{-1}\left(\int|f_\theta|^2(y)W_U(y)dy\right)^2,
		\end{equation*}
		and hence
		\begin{equation*}
			\sum_{\tau_{N-1}}\int_{\mathrm{Narrow}_{N,m}(\tau_{N-1})}\sum_{\theta\subseteq\tilde\tau_{N-1}}|f_{\tilde\theta,m}^\mathcal{B}|^4\lesssim\sum_{U\in\mathcal{G}_\theta}|U|^{-1}\left(\int|f_\theta|^2(y)W_U(y)dy\right)^2
		\end{equation*}
		as claimed.
	\end{proof}

    We have now arranged all the pieces to conclude Proposition \ref{broadnarrow}:

    \begin{proof}[Proof of Proposition \ref{broadnarrow}]

        Immediate from the decomposition \ref{decompofum} and Lemmas \ref{broadbound1p2}, \ref{broadbound2p2}, and \ref{narrowboundp2}.
        
    \end{proof}
	
	\subsection{Reduction to local estimates}
	
	In the above subsections we produced bounds on the measure of the set $U_\alpha=\{x\in B_R:|f(x)|>\alpha\}$. In this subsection we note that, if we can prove Theorem \ref{squarefunction} in the special case that $\{x\in\R^2:|f(x)|>\alpha\}\subseteq Q_R$ for a suitable cube $Q_R$ of radius $R$, then we can conclude that Theorem \ref{squarefunction} is true in the general case.
	
	\begin{proof}[Proof that Prop. \ref{broadnarrow} implies Theorem \ref{squarefunction}]
		
		Fix a $O(1)$-overlapping cover of $\R^2$ by cubes $Q_R$ of radius $R$. Write $\rho_{B_R}$ for a Schwartz function satisfying the following criteria:
		\begin{itemize}
			\item $\rho_{B_R}$ radially symmetric, real, and nonnegative.
			\item $\rho_{B_R}\gtrsim 1_{B_R}$.
			\item $\text{supp}(\hat{\rho}_{B_R})\subseteq B_{2/R}$.
			\item $\sum_{Q_R}\rho_{B_R}(c_{Q_R}-\cdot)\lesssim 1$.
			\item $\rho_{B_R}$ decays rapidly outside of $B_R$.
		\end{itemize}
		For each $Q_R$, write $\rho_{Q_R}=\rho_{B_R}(c_{Q_R}-\cdot)$. By the triangle inequality, there is a subcollection $\Theta$ of the $\theta$ such that, writing $f'=\sum_{\theta\in\Theta}f_\theta$, we have
		\begin{equation*}
			\alpha^4|\{x\in \R^2:|f(x)|>10\alpha\}|\lesssim\alpha^4|\{x\in \R^2:|f'(x)|>\alpha\},
		\end{equation*}
		and such that the $2R^{-1}$-neighborhoods of the $\theta\in\Theta$ are pairwise disjoint. Then $f'\rho_{Q_R}$ has Fourier support in the $\sim R^{-1}$-neighborhood of $\mathbb{P}^1$. By Prop. \ref{broadnarrow}, for each $Q_R'$,
		\begin{equation*}
			\alpha^4|\{x\in Q_R:|f'\rho_{Q_R'}(x)|>\alpha\}|\lesssim\sum_{\substack{R^{-1/2}\leq s\leq 1\\s\text{ dyadic}}}\sum_{\tau:\ell(\tau)=s}\sum_{U\in\mathcal{G}_\tau}|U|^{-1}\|S_U[f'\rho_{Q_R'}]\|_2^4.
		\end{equation*}
		By trivial bounds on $f$ and the rapid decay of $\rho_{B_R}$,
		\begin{equation*}
			\{x\in\R^2:|f'\rho_{Q_R'}(x)|>\alpha\}\subseteq 2Q_R',
		\end{equation*}
		and so
		\begin{equation*}
			\sum_{Q_R}\alpha^4|\{x\in Q_R:|f'\rho_{Q_R'}(x)|>\alpha\}|\lesssim\max_{Q_R}\alpha^4|\{x\in Q_R:|f'\rho_{Q_R'}(x)|>\alpha\}|.
		\end{equation*}
		By Proposition \ref{broadnarrow}, for each $Q_R$,
		\begin{equation*}
			\alpha^4|\{x\in Q_R:|f'\rho_{Q_R'}(x)|>\alpha\}|\lesssim(\log R)^{20}\sum_{\substack{R^{-1/2}\leq s\leq 1\\s\text{ dyadic}}}\sum_{\tau:\ell(\tau)=s}\sum_{U\in\mathcal{G}_\tau}|U|^{-1}\|S_U[f'\rho_{Q_R'}]\|_2^4.
		\end{equation*}
		Adding over all $Q_R'$, we get
		\begin{equation*}
			\sum_{Q_R,Q_R'}\alpha^4|\{x\in Q_R:|f'\rho_{Q_R'}(x)|>\alpha\}|\lesssim(\log R)^{20}\sum_{Q_R'}\sum_{\substack{R^{-1/2}\leq s\leq 1\\s\text{ dyadic}}}\sum_{\tau:\ell(\tau)=s}\sum_{U\in\mathcal{G}_\tau}|U|^{-1}\|S_U[f'\rho_{Q_R'}]\|_2^4.
		\end{equation*}
		If we commute the sum over $Q_R'$ to the inside and use a trivial estimate we conclude
		\begin{equation*}
			\sum_{Q_R,Q_R'}\alpha^4|\{x\in Q_R:|f'\rho_{Q_R'}(x)|>\alpha\}|\lesssim(\log R)^{20}\sum_{\substack{R^{-1/2}\leq s\leq 1\\s\text{ dyadic}}}\sum_{\tau:\ell(\tau)=s}\sum_{U\in\mathcal{G}_\tau}|U|^{-1}\|S_U[\sum_{Q_R'}f'\rho_{Q_R'}]\|_2^4,
		\end{equation*}
		i.e.
		\begin{equation*}
			\sum_{Q_R'}\alpha^4|\{x\in \R^2:|f'\rho_{Q_R'}(x)|>\alpha\}|\lesssim(\log R)^{20}\sum_{\substack{R^{-1/2}\leq s\leq 1\\s\text{ dyadic}}}\sum_{\tau:\ell(\tau)=s}\sum_{U\in\mathcal{G}_\tau}|U|^{-1}\|S_Uf'\|_2^4.
		\end{equation*}
		Finally, by rapid decay,
		\begin{equation*}
			\sum_{Q_R'}\alpha^4|\{x\in \R^2:|f'\rho_{Q_R'}(x)|>\alpha\}|\gtrsim\alpha^4|\{x\in \R^2:|f'(x)|\gtrsim\alpha\}|,
		\end{equation*}
		whereas trivially $S_Uf\geq S_Uf'$ pointwise, so we conclude
		\begin{equation*}
			\alpha^4|\{x\in\R^2:|f(x)|\gtrsim\alpha\}|\lesssim(\log R)^{20}\sum_{\substack{R^{-1/2}\leq s\leq 1\\s\text{ dyadic}}}\sum_{\tau:\ell(\tau)=s}\sum_{U\in\mathcal{G}_\tau}|U|^{-1}\|S_Uf\|_2^4.
		\end{equation*}
		Since this is true for all choices of $\alpha$, we may change variables to conclude
		\begin{equation*}
			\alpha^4|\{x\in\R^2:|f(x)|>\alpha\}|\lesssim(\log R)^{20}\sum_{\substack{R^{-1/2}\leq s\leq 1\\s\text{ dyadic}}}\sum_{\tau:\ell(\tau)=s}\sum_{U\in\mathcal{G}_\tau}|U|^{-1}\|S_Uf\|_2^4,
		\end{equation*}
		as claimed.
		
	\end{proof}
	
	\section{Proof of Theorem \ref{smallcap}}\label{smallcapproof}
	
	In this section we verify that the wave envelope estimate \ref{squarefunction} is strong enough to imply Theorem \ref{smallcap}. This is essentially proven in section 4 of \cite{GM1}, but the latter included $O_\eps(R^\eps)$-lossy pigeonholing steps. Here we perform a more restricted pigeonholing which suffices for the result, and then quote the corresponding incidence geometry calculation in \cite{GM1}.
	
	We will focus on the case $p\geq 4$, where we will have an upper bound for Theorem \ref{smallcap} with power law $R^{\beta(p-\frac{p}{q}-1)-1}$. Under the assumption $\frac{3}{p}+\frac{1}{q}\leq 1$, the remaining case is $3\leq p\leq 4$, where an upper bound $\max\big(1,R^{\beta(\frac{p}{2}-\frac{p}{q})}\big)$ is needed; this is handled in section 4 of \cite{GM1}, and the proof there requires no modification for our purposes.
	
	We begin with the partial decoupling statement in the case $p\geq 4$.
	
	\begin{proposition}\label{partial} Suppose $p\geq 4$ and $\lambda>0$. Let $0\leq k\leq N$ be arbitrary, and fix a canonical scale cap $\tau_k$. Suppose as before that $\Gamma_\beta(R^{-1})$ is a partition of $\mathcal{N}_{R^{-1}}(\mathbb{P}^1)$ into approximate $R^{-\beta}\times R^{-1}$ boxes $\gamma$. Assume $f=\sum_\gamma f_\gamma$ satisfies the following regularity properties:
		\begin{itemize}
			\item[(a)] $\|f_\gamma\|_\infty\leq\lambda$ or $f_\gamma=0$ for each $\gamma$.
			\item[(b)] $\|f_\gamma\|_p^p\leq C^p\lambda^{2-p}\|f_\gamma\|_2^2$ for each $\gamma$ and each $p\geq 1$.
		\end{itemize}
		Write $\gamma_k$ for approximate boxes of dimensions $\sim \max(R^{-\beta},R_k/R)\times R^{-1}$. Then
		\begin{equation}\label{localenvelope}
			\begin{split}
				\sum_{U\in\mathcal{G}_{\tau_k}}|U|\left(\strokedint_U\sum_{\theta\subseteq\tau_k}|f_\theta|^2\right)^2&\lesssim C^p(\log R)^{p-4}
				(\#\tau_k)^{p-4}\alpha^{4-p}\\
				&\times\left(\max_{\gamma_k\subseteq\tau_k}\#(\gamma\subseteq\gamma_k)\times\#(\gamma\subseteq\tau_k)\right)^{\frac{p}{2}-1}\sum_{\gamma\subseteq\tau_k}\|f_\gamma\|_p^p.
			\end{split}
		\end{equation}
		
	\end{proposition}
	\begin{proof}
		
		For each $\theta\subseteq\tau_k$, the small caps $\gamma_k\subseteq\theta$ are either adjacent or are $\sim\max(R^{-\beta},R_k/R)\geq R_k/R$-separated. Fix any $U\in\mathcal{G}_{\tau_k}$.  Since $U\|U_{\tau_k,R}$ has dimensions $R/R_k\times R$, we conclude that the $f_{\gamma_k}$ are locally orthogonal on $U$. Thus
		\begin{equation*}
			\int W_U\sum_{\theta\subseteq\tau_k}|f_{\theta}|^2\lesssim\int W_U\sum_{\gamma_k\subseteq\tau_k}|f_{\gamma_k}|^2,
		\end{equation*}
		and so, appealing to the definition of $\mathcal{G}_{\tau_k}$,
		\begin{equation*}
			\frac{\alpha^2}{(\#\tau_k)^2}\lesssim (\log R)^8|U|^{-1}\int W_U\sum_{\gamma_k\subseteq\tau_k}|f_{\gamma_k}|^2,
		\end{equation*}
		where we have suppressed the dependence on $C_{\mathfrak{p}}$. Multiplying the left-hand side of \ref{localenvelope} by the $(\frac{p}{2}-2)$-power of the latter display, we obtain the estimate
		\begin{equation}\label{envest}
			\sum_{U\in\mathcal{G}_{\tau_k}}|U|\left(\strokedint_U\sum_{\theta\subseteq\tau_k}|f_\theta|^2\right)^2\lesssim(\#\tau_k)^{p-4}\alpha^{4-p}(\log R)^{p-4}\sum_{U\in\mathcal{G}_{\tau_k}}|U|\left(\strokedint_U\sum_{\gamma_k\subseteq\tau_k}|f_{\gamma_k}|^2\right)^{\frac{p}{2}}.
		\end{equation}
		Uniformity assumption (a) implies
		\begin{equation*}
			\Big\|\sum_{\gamma_k\subseteq\tau_k}|f_{\gamma_k}|^2\Big\|_\infty\lesssim \lambda^2\big[\max_{\gamma_k\subseteq\tau_k}\#(\gamma\subseteq\gamma_k)\big]\times\#(\gamma\subseteq\tau_k).
		\end{equation*}
		By removing factors of $\|\sum_{\gamma_k\subseteq\tau_k}|f_{\gamma_k}|^2\|_\infty$ from \ref{envest}, we obtain
		\begin{equation*}
			\begin{split}
				\sum_{U\in\mathcal{G}_{\tau_k}}|U|\left(\strokedint_U\sum_{\theta\subseteq\tau_k}|f_\theta|^2\right)^2&\lesssim (\#\tau_k)^{p-4}\alpha^{4-p}(\log R)^{p-4} \lambda^{p-2}\\
				&\times\left(\max_{\gamma_k\subseteq\tau_k}\#(\gamma\subseteq\gamma_k)\times\#(\gamma\subseteq\tau_k)\right)^{\frac{p}{2}-1}\\
				&\times\sum_{U\in\mathcal{G}_{\tau_k}}\int W_U\sum_{\gamma_k\subseteq\tau_k}|f_{\gamma_k}|^2,
			\end{split}
		\end{equation*}
		and by local orthogonality and uniformity assumption (b)
		\begin{equation*}
			\sum_{U\in\mathcal{G}_{\tau_k}}\int W_U\sum_{\gamma_k\subseteq\tau_k}|f_{\gamma_k}|^2\lesssim\int\sum_{\gamma\subseteq\tau_k}|f_\gamma|^2\lesssim C^p \lambda^{2-p}\sum_{\gamma\subseteq\tau_k}\|f_\gamma\|_p^p.
		\end{equation*}
		Together we get the estimate
		\begin{equation*}
			\begin{split}
				\sum_{U\in\mathcal{G}_{\tau_k}}|U|\left(\strokedint_U\sum_{\theta\subseteq\tau_k}|f_\theta|^2\right)^2&\lesssim C^p(\log R)^{p-4}
				(\#\tau_k)^{p-4}\alpha^{4-p} \\
				&\times\left(\max_{\gamma_k\subseteq\tau_k}\#(\gamma\subseteq\gamma_k)\times\#(\gamma\subseteq\tau_k)\right)^{\frac{p}{2}-1}\sum_{\gamma\subseteq\tau_k}\|f_\gamma\|_p^p,
			\end{split}
		\end{equation*}
		as claimed.
	\end{proof}
	
	\begin{remark}
		Suppose that $p=2+2/\beta$ and $q=p$. Plugging in the bounds $\#\tau_k\leq R_k$, $\max_{\gamma_k\subseteq\tau_k}\#(\gamma\subseteq\gamma_k)\leq\max(1,R^{\beta-1}R_k)$, $\#(\gamma\subseteq\tau_k)\leq R_k^{-1}R^\beta$, and $R_k\leq R^{1/2}$, this immediately implies the estimate
		\begin{equation*}
			\sum_{U\in\mathcal{G}_{\tau_k}}|U|\left(\strokedint_U\sum_{\theta\subseteq\tau_k}|f_\theta|^2\right)^2\lesssim(\log R)^{p-4}\alpha^{4-p}R^{\beta(p-2)-1}\sum_{\gamma\subseteq\tau_k}\|f_\gamma\|_p^p,
		\end{equation*}
		and hence, by Theorem \ref{squarefunction},
		\begin{equation*}
			\alpha^{p}|U_\alpha|\lesssim(\log R)^{17+p}R^{\beta(p-2)-1}\sum_\gamma\|f_\gamma\|_p^p
		\end{equation*}
		as claimed. It essentially remains to remove assumptions (a) and (b) above, and to track the dependence on $q$.
	\end{remark}
	
	\begin{proof}[Proof of Theorem \ref{smallcap}]
		
		Consider the decomposition
		\begin{equation*}
			f=\sum_{\gamma\in\Gamma_\beta(R^{-1})}f_\gamma\,\,.
		\end{equation*}
		By scaling we may assume that $\max_\theta\|f_\theta\|_\infty=1$. Then we may write
		\begin{equation*}
			f=\sum_{(\log R)^{-2}R^{-1/2}\lesssim\lambda\lesssim 1}\sum_{\substack{\gamma\in\Gamma_\beta(R^{-1})\\\|f_\gamma\|_\infty\sim\lambda}}f_\gamma+R^{-1000}\eta,
		\end{equation*}
		where the $\lambda$ range over dyadic numbers, and $\eta$ is rapidly decaying outside of $B_R$. We abbreviate
		\begin{equation*}
			\Gamma_\beta^\lambda(R^{-1})=\{\gamma\in\Gamma_\beta(R^{-1}):\|f_\gamma\|_\infty\sim\lambda\}.
		\end{equation*}
		
		Then, for each $\lambda$, consider the wave envelope expansion
		\begin{equation*}
			\sum_{\gamma\in\Gamma_\beta^\lambda(R^{-1})}f_\gamma=\sum_{\gamma\in\Gamma_\beta^\lambda(R^{-1})}\sum_U\psi_Uf_\gamma,
		\end{equation*}
		where each $U$ has dimensions $\sim R^\beta\times R$ and has long edge parallel to $\mathbf{n}_{c_\gamma}$. Since $\gamma\in\Gamma_\beta^\lambda(R^{-1})$, there is some $U$ such that $\|\psi_Uf\|_\infty\sim\lambda$. If we write $\mathcal{U}_\lambda=\mathcal{U}_\lambda^\gamma$ for the set of $U$ for which $\|\psi_Uf_\gamma\|_\infty\sim\lambda$, then for all $\gamma\in\Gamma_\beta^{\lambda}(R^{-1})$
		\begin{equation*}
			\Big\|\sum_{U\in\mathcal{U}_\lambda}\psi_Uf_\gamma\Big\|_p^p\sim_p(\#\mathcal{U}_\lambda)|U|\lambda^p,
		\end{equation*}
		and so
		\begin{equation*}
			\Big\|\frac{1}{\lambda}\sum_{U\in\mathcal{U}_\lambda}\psi_Uf_\gamma\Big\|_p^p\sim_p(\#\mathcal{U}_\lambda)|U|\sim \Big\|\frac{1}{\lambda}\sum_{U\in\mathcal{U}_\lambda}\psi_Uf_\gamma\Big\|_2^2.
		\end{equation*}
		For each $1\leq\mathfrak{r}\leq R$ dyadic and each $\lambda$, write $\Gamma_\beta^{\lambda;\mathfrak{r}}(R^{-1})$ to be the collection of $\gamma\in\Gamma_\beta^{\lambda}(R^{-1})$ such that $\#\mathcal{U}_\lambda^\gamma\sim\mathfrak{r}$. Define for $\gamma\in\Gamma_\beta^\lambda(R^{-1})$
		\begin{equation*}
			g_\gamma^{(\lambda)}=\frac{1}{\lambda}\sum_{U\in\mathcal{U}_\lambda}\psi_Uf_\gamma
		\end{equation*}
		and
		\begin{equation*}
			g^{(\lambda,\mathfrak{r})}=\sum_{\gamma\in\Gamma_\beta^{\lambda;\mathfrak{r}}(R^{-1})} g_\gamma^{(\lambda)}.
		\end{equation*}
		Then for each $\lambda,\mathfrak{r}$, and $\alpha>0$ we have
		\begin{equation}\label{ineq:pre-dec}
			\alpha^4|\{x:|\lambda g^{(\lambda,\mathfrak{r})}(x)|>\alpha\}|\lesssim(\log R)^{20}\sum_{1\leq k\leq N}\sum_{\tau_k}\sum_{U\in\mathcal{G}_{\tau_k}[\lambda g^{(\lambda,\mathfrak{r})};\alpha]}|U|\left(\strokedint_U\sum_{\theta\subseteq\tau_k}|\lambda g_\theta^{(\lambda,\mathfrak{r})}|^2\right)^2,
		\end{equation}
		where we have written $\mathcal{G}_{\tau_k}[\lambda g^{(\lambda,\mathfrak{r})};\alpha]$ to record that the pruning is of $\lambda g^{(\lambda,\mathfrak{r})}$ with respect to the parameter $\alpha$.
        
        Let $1\leq k\leq N$ be arbitrary. If $R_k>R^{1-\beta}$, then we let $1\leq \frak t\leq R_k/R^{1-\beta}$ be arbitrary; in the alternative case we insist $\frak t=1$. Write $\Gamma_{k}^{\lambda;\frak r,\frak t}$ be the collection of $\gamma_k$ such that
        \begin{equation*}
            0<\#\{\gamma\in\Gamma_\beta^{\lambda;\frak r}(R^{-1}):\gamma\subseteq\gamma_k\}\sim\frak t.
        \end{equation*}For each  $1\leq\mathfrak{s}\leq R^{1/2}$, let $\mathcal{T}_k(\mathfrak{s},\frak t)$ denote the collection of $\tau_k$ such that
        \begin{equation*}
            0<\#\left\{\gamma\in\Gamma_\beta^{\lambda;\frak r}:\gamma\subseteq\bigcup_{\substack{\gamma_k\subseteq\tau_k\\\gamma_k\in\Gamma_k^{\lambda;\frak r,\frak t}}}\gamma_k\right\}\sim\frak s.
        \end{equation*}
        By Prop. \ref{partial} we have
        \begin{equation*}
            \begin{split}
                \sum_{\tau_k\in\cal T_k(\frak s,\frak t)}&\sum_{U\in\cal G_{\tau_k}[\lambda g^{(\lambda,\mathfrak{r})};\alpha]}|U|\left(\strokedint_U\sum_{\theta\subseteq\tau_k}|\lambda g_\theta^{(\lambda,\mathfrak{r})}|^2\right)^2\\
                &\lesssim_p(\log R)^{p-4}\#\cal T_k(\frak s,\frak t)^{p-4}(\frak s\frak t)^{\frac{p}{2}-1}\sum_{\tau_k\in\cal T_k(\frak s,\frak t)}\alpha^{4-p}\sum_{\gamma\subseteq\tau_k}\|\lambda g_\gamma^{(\lambda,\frak r)}\|_p^p.
            \end{split}
        \end{equation*}
        The remainder of the analysis is straightforward caseword, virtually identical to \cite{GM1}; we include it for completeness. If $3\leq p\leq 4$ and $\alpha^2\leq\frak s\times\cal T_k(\frak s,\frak t)$, then as in the proof of Theorem 5 in \cite{GM1}, case 1, we have
        \begin{equation*}
            \alpha^4|\{x:|\lambda g^{(\lambda,\mathfrak{r})}(x)|>\alpha\}|\leq(s\times\cal T_k(\frak s,\frak t))^{\frac{p}{2}}\max_{\gamma\in\Gamma_\beta^{\lambda;\frak r}}\|\lambda g_\gamma^{(\lambda,\frak r)}\|_2^2\lesssim_p R^{\beta(\frac{p}{2}-\frac{p}{q})}(s\times\cal T_k(\frak s,\frak t))^{\frac{p}{q}}\max_{\gamma\in\Gamma_\beta^{\lambda;\frak r}}\|\lambda g_\gamma^{(\lambda,\frak r)}\|_p^p,
        \end{equation*}
        and the result easily follows. If instead $\alpha^2>\frak s\times\cal T_k(\frak s,\frak t)$, then using the inequality \ref{ineq:pre-dec} and the inequality
        \begin{equation*}
            \alpha^{p-4}\#\cal T_k(\frak s,\frak t)\frak s^2\frak t\lesssim R^{\beta(\frac{p}{2}-\frac{p}{q})}(\frak s\times\#\cal T_k(\frak s,\frak t))^{\frac{p}{q}}
        \end{equation*}
        (see case 1 in the proof of Theorem 5 in \cite{GM1}), the desired result follows.
        
        If $4\leq p\leq 6$, then case 2 of the proof\footnote{That case uses a hidden assumption of $R_k>R^{1-\beta}$; a similar argument handles the complementary case.} of Theorem 5 in \cite{GM1} implies the inequality
        \begin{equation}\label{ineq:triv_input}
            \#\cal T_k(\frak s,\frak t)^{p-4}(\frak s\frak t)^{\frac{p}{2}-1}\lesssim R^{\beta(p-\frac{p}{q}-1)-1}\big(\frak s\times\#\cal T_k(\frak s,\frak t)\big)^{\frac{p}{q}-1}.
        \end{equation}
        If $p>6$, then case 3 of the same proof shows inequality \ref{ineq:triv_input} again. 
        It follows that
        \begin{equation*}
            \sum_{\tau_k\in\cal T_k(\frak s,\frak t)}\sum_{U\in\cal G_{\tau_k}[\lambda g^{(\lambda,\mathfrak{r})};\alpha]}|U|\left(\strokedint_U\sum_{\theta\subseteq\tau_k}|\lambda g_\theta^{(\lambda,\mathfrak{r})}|^2\right)^2\lesssim (\log R)^{p-4}\alpha^{4-p}R^{\beta(p-\frac{p}{q}-1)-1}\Big(\sum_{\gamma\subseteq\bigcup\cal T_k(\frak s,\frak t)}\|\lambda g_\gamma^{(\lambda,\frak r)}\|_p^q\Big)^{\frac{p}{q}},
        \end{equation*}
        and hence, by H\"older inequalities and a layer-cake integral, we obtain
        \begin{equation*}
            \|\lambda g^{(\lambda,\frak r)}\|_p^p\lesssim (\log R)^{20+p}R^{\beta(p-\frac{p}{q}-1)-1}\Big(\sum_{\gamma}\|\lambda g_\gamma^{(\lambda,\frak r)}\|_p^q\Big)^{\frac{p}{q}}.
        \end{equation*}
        Recalling the sum $f\approx\sum_{\lambda,\frak r}\lambda g^{(\lambda,\frak r)}$, and that the latter summands are refinements of partitions of unity applied to $f$, we obtain the conclusion
        \begin{equation*}
            \|f\|_p^p\lesssim(\log R)^{18+3p}R^{\beta(p-\frac{p}{q}-1)-1}\Big(\sum_{\gamma}\|f_\gamma\|_p^q\Big)^{\frac{p}{q}}.
        \end{equation*}

	\end{proof}
    
	\section{Appendix: Proofs of square function lemmas}\label{section:appendix}
	
	In this appendix, we record the proofs of the critical lemmas for the high/low method in Fourier analysis that are appropriate for our sequence of scales. The proofs are essentially identical to those in \cite{GM1}, but we record them for completeness, in addition to verifying that the losses are as claimed.
	
	\begin{lemma}[Pointwise local constancy lemmas]
		\begin{enumerate}
			\item[(a)] For any $\theta$, $|f_\theta|^2\lesssim|f_\theta|^2*|\rho_\theta^\vee|$.
			\item[(b)] For any $k,m$ and any $x$,
			\begin{equation*}
				|f_{m,\tau_k}|^2(x)\lesssim |f_{m,\tau_k}|^2*w_{R_k}(x).
			\end{equation*}
		\end{enumerate}
	\end{lemma}
	\begin{proof}
		(a): Note first that
		\begin{equation*}
			|f_\theta|^2(y)=|f_\theta*\rho_\theta^\vee|^2(y)\leq\left|\int_\R |f_\theta|(z)|\rho_\theta|^{1/2}(y-z)|\rho_\theta|^{1/2}(y-z)dz\right|^2,
		\end{equation*}
		by considering the Fourier support. By H\"older,
		\begin{equation*}
			\left|\int_\R |f_\theta|(z)|\rho_\theta|^{1/2}(y-z)|\rho_\theta|^{1/2}(y-z)dz\right|^2\leq\|\rho_\theta^\vee\|_1\left(|f_\theta|^2*|\rho_\theta^\vee|\right)(y).
		\end{equation*}
		Note that, by change-of-variable, $\|\rho_\theta^\vee\|_1=\|\rho_{0}^\vee\|_1=O(1)$ independent of $R$. Thus
		\begin{equation*}
			|f_\theta|^2\lesssim|f_\theta|^2*|\rho_\theta^\vee|
		\end{equation*}
		as claimed.
		
		(b): By the pruning lemma, $|f_{m,\tau_k}|^2$ has Fourier support contained in $\bigcup_{\theta,\theta'\subseteq\tau_k}(N-m+2)(\theta-\theta')$, which is in turn contained in the set $B_{\frac{1}{2}(\log R)R_k^{-1}}$. Let $\rho_k$ be a real smooth radially symmetric cutoff function that is equal to $1$ on $B_{\frac{1}{2}(\log R)R_k^{-1}}$ and is supported in $B_{(\log R)R_k^{-1}}$. By the same calculation as in (a),
		\begin{equation*}
			|f_{m,\tau_k}|^2=|f_{m,\tau_k}|^2*\rho_k^\vee\lesssim|f_{m,\tau_k}|^2*|\rho_k^\vee|.
		\end{equation*}
		On the other hand, we clearly have $|\rho_k^\vee|\lesssim w_{R_k}$, and we are done.
		
	\end{proof}
	
	\begin{lemma}[Wave envelope expansion]
		\begin{enumerate}
			\item[(a)] If $r\lesssim (\log R)R_{k}/R$, then
			\begin{equation*}
				\int\Big|\sum_{\theta\subseteq\tau_k}|f_{m,\theta}^\mathcal{B}|^2*\eta_{\sim r}^\vee\Big|^2\lesssim \int\Big| \sum_{\theta\subseteq\tau_k}|f_{m,\theta}^\mathcal{B}|^2*|\rho_{C(\log R)U_{\tau_k,R}^*}^\vee|\,\Big|^2,
			\end{equation*}
            where $U_{\tau_k,R}^*$ is a rectangle of dimensions $R_k/R\times R^{-1}$ centered at the origin with long edge parallel to $\mathbf{t}_{c_{\tau_k}}$.
			\item[(b)] If $k\geq m$, then
			\begin{equation*}
				\int\Big|\sum_{\theta\subseteq\tau_k}|f_{m,\theta}^\mathcal{B}|^2*|\rho_{(\log R)U_{\tau_k,R}^*}^\vee|\,\Big|^2\lesssim(\log R)^2\sum_{U\in\mathcal{G}_{\tau_k}}|U|\left(\strokedint_U\sum_{\theta\subseteq\tau_k}|f_\theta|^2\right)^2+R^{-100}.
			\end{equation*}
		\end{enumerate}
		
	\end{lemma}
	\begin{proof}
		(a): The Fourier support of $\sum_{\theta\subseteq\tau_k}|f_{m,\theta}^\mathcal{B}|^2*\eta_{\sim r}^\vee$ is contained in the set
		\begin{equation*}
			(N-m+2)\bigcup_{\theta\subseteq\tau_k}(\theta-\theta)\cap B_{2r}\subseteq C(\log R)U_{\tau_k,R}^*.
		\end{equation*}Thus
		\begin{equation*}
			\begin{split}
				\int\Big|\sum_{\theta\subseteq\tau_k}|f_{m,\theta}^\mathcal{B}|^2*\eta_{\sim r}^\vee\Big|^2&=\int\Big|\sum_{\theta\subseteq\tau_k}\widehat{|f_{m,\theta}^\mathcal{B}|^2}\eta_{\sim r}\Big|^2\\
				&\lesssim\int\Big|\sum_{\theta\subseteq\tau_k}\widehat{|f_{m,\theta}^\mathcal{B}|^2}\rho_{C(\log R)U_{\tau_k,R}^*}\Big|^2\\
				&\leq\int\Big|\sum_{\theta\subseteq\tau_k}|f_{m,\theta}^\mathcal{B}|^2*|\rho_{C(\log R)U_{\tau_k,R}^*}^\vee|\Big|^2
			\end{split}
		\end{equation*}
		as claimed.
		
		(b): Since $k\geq m$, $|f_{m,\theta}^\mathcal{B}|\leq|f_{k,\theta}|\leq|f_{k+1,\theta}|\leq|f_\theta|$ by the pruning lemmas, so
		\begin{equation*}
			\int\Big|\sum_{\theta\subseteq\tau_k}|f_{m,\theta}^\mathcal{B}|^2*|\rho_{C(\log R)U_{\tau_k,R}^*}^\vee|\Big|^2\leq\int\Big|\sum_{\theta\subseteq\tau_k}|f_{k,\theta}|^2*|\rho_{C(\log R)U_{\tau_k,R}^*}^\vee|\Big|^2.
		\end{equation*}
		By the definition of the pruning,
		\begin{equation*}
			\int\Big[\sum_{\theta\subseteq\tau_k}|f_{k,\theta}|^2*|\rho_{C(\log R)U_{\tau_k,R}^*}^\vee|\Big]^2=\int\Big[\sum_{\theta\subseteq\tau_k}\int|\sum_{U\in\mathcal{G}_{\tau_k}}\psi_Uf_{k+1,\theta}(y)|^2|\rho_{C(\log R)U_{\tau_k,R}^*}^\vee|(x-y)dy\Big]^2dx,
		\end{equation*}
		which, since $\sum_{U\in\mathcal{G}_{\tau_k}}\psi_U\leq 1$, may be bounded from above by
		\begin{equation*}
			\int\Big[\sum_{\theta\subseteq\tau_k}\int\sum_{U\in\mathcal{G}_{\tau_k}}\psi_U(y)|f_{\theta}(y)|^2|\rho_{C(\log R)U_{\tau_k,R}^*}^\vee|(x-y)dy\Big]^2dx.
		\end{equation*}
		We may remove the $\psi_U$ from the $dy$ integral by replacing it with $\widetilde{\psi}_U(x)=\max_{z\in x+(\log R)^{1.5}U_{\tau_k,R}}|\psi_U(z)|$; note that for each $y$ and $x\in y+U_{\tau_k,R}$ we have $\psi_U(y)\leq\widetilde{\psi}_U(x)$. Thus
		\begin{equation*}
			\begin{split}
				\int\Big[\sum_{\theta\subseteq\tau_k}&\int\sum_{U\in\mathcal{G}_{\tau_k}}\psi_U(y)|f_{\theta}(y)|^2|\rho_{C(\log R)U_{\tau_k,R}^*}^\vee|(x-y)dy\Big]^2dx\\
				&\leq\int\Big[\sum_{U\in\mathcal{G}_{\tau_k}}\widetilde{\psi}_U(x)\sum_{\theta\subseteq\tau_k}\int_{x+(\log R)^{1.5}U_{\tau_k,R}}|f_{\theta}(y)|^2|\rho_{C(\log R)U_{\tau_k,R}^*}^\vee|(x-y)dy\Big]^2dx\\
				&+\int\Big[\sum_{\theta\subseteq\tau_k}\int_{\R^2\setminus(x+(\log R)^{1.5}U_{\tau_k,R})}\sum_{U\in\mathcal{G}_{\tau_k}}\psi_U(y)|f_{\theta}(y)|^2|\rho_{C(\log R)U_{\tau_k,R}^*}^\vee|(x-y)dy\Big]^2dx\\
				&=:(I)+(II).
			\end{split}
		\end{equation*}
		Note that $|\rho_{C(\log R)U_{\tau_k,R}^*}^\vee|$ decays almost-exponentially outside of $(\log R)^{-1}U_{\tau_k,R}$, so when $y\not\in x+(\log R)^{1.5}U_{\tau_k,R}$ we have $|\rho_{C(\log R)U_{\tau_k,R}^*}^\vee(x-y)|\lesssim R^{-100}$. By Minkowski, $(II)$ may be controlled via
		\begin{equation*}
			\begin{split}
				&\int\left[\sum_{\theta\subseteq\tau_k}\int_{\R^2\setminus(x+(\log R)^{1.5}U_{\tau_k,R})}\sum_{U\in\mathcal{G}_{\tau_k}}\psi_U(y)|f_{\theta}(y)|^2|\rho_{C(\log R)U_{\tau_k,R}^*}^\vee|(x-y)dy\right]^2dx\\
				&\leq\left(\int\sum_{\theta\subseteq\tau_k}\sum_{U\in\mathcal{G}_{\tau_k}}\psi_U(y)|f_{\theta}(y)|^2\left[\int_{\R^2\setminus(y+(\log R)^{1.5}U_{\tau_k,R})}|\rho_{C(\log R)U_{\tau_k,R}^*}^\vee|^2(x-y)dx\right]^{1/2}dy\right)^2\\
				&\lesssim R^{-200}\left(\sum_{U\in\mathcal{G}_{\tau_k}}\int\psi_U(y)\sum_{\theta\subseteq\tau_k}|f_{\theta}(y)|^2dy\right)^2\leq R^{-100}.
			\end{split}
		\end{equation*}
		
		On the first integral $(I)$, we may estimate
		\begin{equation*}
			\begin{split}
				&\int\left[\sum_{U\in\mathcal{G}_{\tau_k}}\widetilde{\psi}_U(x)\sum_{\theta\subseteq\tau_k}\int_{x+(\log R)^{1.5}U_{\tau_k,R}}|f_{\theta}(y)|^2|\rho_{C(\log R)U_{\tau_k,R}^*}^\vee|(x-y)dy\right]^2dx\\
				&\lesssim\int\sum_{U\in\mathcal{G}_{\tau_k}}\widetilde{\psi}_U(x)\left[\sum_{\theta\subseteq\tau_k}\int_{x+(\log R)^{1.5}U_{\tau_k,R}}|f_{\theta}(y)|^2|\rho_{C(\log R)U_{\tau_k,R}^*}^\vee|(x-y)dy\right]^2dx.
			\end{split}
		\end{equation*}
		By Minkowski,
		\begin{equation*}
			\begin{split}
				\sum_{U\in\mathcal{G}_{\tau_k}}\int\widetilde{\psi}_U(x)&\left[\sum_{\theta\subseteq\tau_k}\int|f_{\theta}(y)|^2|\rho_{C(\log R)U_{\tau_k,R}^*}^\vee|(x-y)dy\right]^2dx\\
				&\leq\sum_{U\in\mathcal{G}_{\tau_k}}\left(\int\sum_{\theta\subseteq\tau_k}|f_\theta|^2(y)\left[\int\widetilde{\psi}_U(x)|\rho_{C(\log R)U_{\tau_k,R}^*}^\vee|^2(x-y)dx\right]^{1/2}dy\right)^2\\
				&\lesssim(\log R)^2\sum_{U\in\mathcal{G}_{\tau_k}}|U|
				^{-1}\left(\int W_U\sum_{\theta\subseteq\tau_k}|f_\theta|^2(y)dy\right)^2.
			\end{split}
		\end{equation*}
		
		We conclude that
		\begin{equation*}
			\int\Big|\sum_{\theta\subseteq\tau_k}|f_{m,\theta}^\mathcal{B}|^2*|\rho_{\tau_k}^\vee|\Big|^2\lesssim(\log R)^2\sum_{U\in\mathcal{G}_{\tau_k}}|U|
			^{-1}\left(\int W_U\sum_{\theta\subseteq\tau_k}|f_\theta|^2(y)dy\right)^2+R^{-100},
		\end{equation*}
		as claimed.
	\end{proof}

	\begin{lemma}[Replacement lemma] $|f(x)-f_N(x)|\lesssim\frac{\alpha}{C_{\mathfrak{p}}^{1/2}(\log R)^8}$.
	\end{lemma}
	\begin{proof}
		Consider the difference
		\begin{equation*}
			|f(x)-f_N(x)|\leq\sum_\theta\sum_{U\not\in\mathcal{G}_\theta}\psi_U(x)|f_\theta(x)|.
		\end{equation*}
		By an analogue of local constancy (a),
		\begin{equation*}
			|\psi_Uf_\theta|\lesssim\left(|\psi_Uf_\theta|^2*|\rho_{2\theta}^\vee|\right)^{1/2},
		\end{equation*}
		so
		\begin{equation*}
			\begin{split}
				|f(x)-f_N(x)|&\leq\sum_\theta\sum_{U\not\in\mathcal{G}_\theta}\left(|\psi_Uf_\theta|^2*|\rho_{2\theta}^\vee|\right)^{1/2}\\
				&=\sum_\theta\sum_{U\not\in\mathcal{G}_\theta}\left(\int|\psi_Uf_\theta|^2(y)|\rho_{2\theta}^\vee|(x-y)dy\right)^{1/2}.
			\end{split}
		\end{equation*}
		Next, since $\psi_U\lesssim W_U$,
		\begin{equation*}
			\begin{split}
				|f(x)-f_N(x)|&\leq\sum_\theta\sum_{U\not\in\mathcal{G}_\theta}\left(\int W_U(y)|f_\theta|^2(y)\psi_U(y)|\rho_{2\theta}^\vee|(x-y)dy\right)^{1/2}\\
				&\leq\sum_\theta\max_{U\not\in\mathcal{G}_\theta}\left(\int W_U(y)|f_\theta|^2(y)dy\right)^{1/2}\\
				&\,\,\times\sum_{U\not\in\mathcal{G}_\theta}\left(\sup_y\psi_U(y)|\rho_{2\theta}^\vee|(x-y)\right)^{1/2}
			\end{split}
		\end{equation*}
		by Cauchy-Schwarz. By the rapid decay of $\psi_U$ outside of $U$ and local constancy of $\rho_{2\theta}^\vee$,
		\begin{equation*}
			\begin{split}
				\sum_{U\not\in\mathcal{G}_\theta}\|\psi_U(\cdot)\rho_{2\theta}^\vee(x-\cdot)\|_{L^\infty(\R^2)}&\lesssim\sum_U\|\rho_{2\theta}^\vee(x-\cdot)\|_{L^\infty(U)}\\
				&\lesssim|U|^{-1}\sum_U\|\rho_{2\theta}^\vee(x-\cdot)\|_{L^1(U)}\\
				&=|U|^{-1}\|\rho_{2\theta}^\vee\|_{L^1(\R^2)}\\
				&\lesssim|U|^{-1},
			\end{split}
		\end{equation*}
		so that
		\begin{equation*}
			|f(x)-f_N(x)|\lesssim|U|^{-1/2}\sum_{\theta}\max_{U\not\in\mathcal{G}_\theta}\left(\int W_U|f_\theta|^2(y)dy\right)^{1/2}.
		\end{equation*}
		Finally, by the definition of $\mathcal{G}_\theta$,
		\begin{equation*}
			|f(x)-f_N(x)|\lesssim\sum_{\theta}\max_{U\not\in\mathcal{G}_\theta}\frac{\alpha}{(\#\theta)C_{\mathfrak{p}}^{1/2}(\log R)^8}=\frac{\alpha}{C_\mathfrak{p}^{1/2}(\log R)^8}
		\end{equation*}
		as claimed.
	\end{proof}
    
	\begin{lemma}[Low lemma]
		For any $2\leq m\leq k\leq N$, $0\leq s\leq k$, and $r\leq (\log R)R_k^{-1}$,
		\begin{equation*}
			|f_{m,\tau_{s}}^{\mathcal{B}}|^2*\eta_{\leq r}^\vee(x)=\sum_{\tau_k\subseteq\tau_s}\sum_{\tau_k':\tau_k\text{ near }\tau_k'}\Big(f_{m,\tau_k}^{\mathcal{B}}\overline{f_{m,\tau_k'}^{\mathcal{B}}}\Big)*\eta_{\leq r}^\vee(x)
		\end{equation*}
		for any $x$ and any $\tau_s$.
	\end{lemma}
	\begin{proof}
		
		Indeed,
		\begin{equation*}
			\begin{split}
				|f_{m,\tau_{s}}^{\mathcal{B}}|^2*\eta_{\leq r}^\vee(x)&=\int_{\R^2}|f_{m,\tau_{s}}^{\mathcal{B}}|^2(x-y)\eta_{\leq r}^\vee(y)dy\\
				&=\int_{\R^2}e^{2\pi ix\cdot\xi}\Big[\widehat{f_{m,\tau_{s}}^{\mathcal{B}}}*\widehat{\overline{f_{m,\tau_{s}}^{\mathcal{B}}}}(\xi)\Big]\eta_{\leq r}(\xi)d\xi\\
				&=\sum_{\tau_k,\tau_k'\subseteq\tau_{s}}\int_{\R^2}e^{2\pi ix\cdot\xi}\Big[\widehat{f_{m,\tau_{k}}^{\mathcal{B}}}*\widehat{\overline{f_{m,\tau_{k}'}^{\mathcal{B}}}}(\xi)\Big]\eta_{\leq r}(\xi)d\xi.
			\end{split}
		\end{equation*}
		Note that each $\widehat{f_{m,\tau_k}^{\mathcal{B}}}$ has support in the set $\bigcup_{\theta\subseteq\tau_k}(N-m+2)\theta\subseteq(N-m+2)\tau_k$; thus the convolution in the latter integrand is supported in the set $(N-m+2)(\tau_k-\tau_k')\subseteq(\log R)(\tau_k-\tau_k')$, which is contained in the ball $B_{C(\log R)R_k^{-1}}(c_{\tau_k}-c_{\tau_k'})$ for suitable universal constant $C$. Since $\eta_{\leq r}$ has support in the ball of radius $2(\log R)R_k^{-1}$, and the diameter of each $\tau_k$ is $\sim R_k^{-1}$, we conclude that for each $\tau_k$ there are $\lesssim\log R$-many neighboring $\tau_k'$ such that the support of $\widehat{f_{m,\tau_{k}}^{\mathcal{B}}}*\widehat{\overline{f_{m,\tau_{k}'}^{\mathcal{B}}}}$ has nontrivial intersection with the support of $\eta_{\leq r}$. Thus
		\begin{equation*}
			\begin{split}
				\sum_{\tau_k,\tau_k'\subseteq\tau_{s}}\int_{\R^2}e^{-2\pi ix\cdot\xi}&\Big[\widehat{f_{m,\tau_{k}}^{\mathcal{B}}}*\widehat{\overline{f_{m,\tau_{k}'}^{\mathcal{B}}}}(\xi)\Big]\eta_{\leq r}(\xi)d\xi\\
				&=\sum_{\substack{\tau_k,\tau_k'\subseteq\tau_{s}\\\tau_k\text{ near }\tau_k'}}\int_{\R^2}e^{-2\pi ix\cdot\xi}\Big[\widehat{f_{m,\tau_{k}}^{\mathcal{B}}}*\widehat{\overline{f_{m,\tau_{k}}^{\mathcal{B}}}}(\xi)\Big]\eta_{\leq r}(\xi)d\xi.
			\end{split}
		\end{equation*}
		By Plancherel again, we conclude.
	\end{proof}
	
	\begin{lemma}[High Lemmas] For any $m,k,s$, and $\ell$ such that $2\leq m\leq N$, $0\leq s\leq k$, and $k+\ell\leq N$, and any cap $\tau_s$,
		\begin{enumerate}
			\item[(a)] 
			\begin{equation*}
				\int\Big|\sum_{\theta\subseteq\tau_s}|f_{m,\theta}^{\mathcal{B}}|^2*\eta_{\geq R_k/R}^\vee\Big|^2\lesssim\log R\sum_{\tau_k\subseteq\tau_s}\int\Big|\sum_{\theta\subseteq\tau_k}|f_{m,\theta}|^2*\eta_{\geq R_k/R}^\vee\Big|^2,
			\end{equation*}
			\item[(b)]
			\begin{equation*}
				\int\Big|\sum_{\tau_k\subseteq\tau_s}|f_{m,\tau_k}^{\mathcal{B}}|^2*\eta_{\geq R_k^{-1}}^\vee\Big|^2\lesssim(\log R)\sum_{\tau_k\subseteq\tau_s}\int|f_{m,\tau_k}^{\mathcal{B}}|^4,
			\end{equation*}
			\item[(c)]
			\begin{equation*}
				\int\Big|\sum_{\tau_{k}\subseteq\tau_s}\sum_{\tau_k'\,\mathrm{ near }\,\tau_k}(f_{m,\tau_{k}}^\mathcal{B}\overline{f_{m,\tau_{k}'}^\mathcal{B}})*\eta_{\geq R_{k+\ell}^{-1}}^\vee\Big|^2\lesssim(\log R)^{3}R_\ell\sum_{\tau_k\subseteq\tau_s}\int|f_{m,\tau_k}^\mathcal{B}|^4.
			\end{equation*}
		\end{enumerate}
	\end{lemma}
	\begin{proof}
        There is no loss of generality in taking $s=0$, so $\tau_s$ is trivial.
    
		(a): By Plancherel,
		\begin{equation*}
			\int\Big|\sum_\theta|f_{m,\theta}^{\mathcal{B}}|^2*\eta_{\geq R_k/R}^\vee\Big|^2=\int_{|\xi|\geq R_k/R}\Big|\sum_{\tau_k}\sum_{\theta\subseteq\tau_k}\widehat{|f_{m,\theta}^{\mathcal{B}}|^2}(\xi)\eta_{\geq R_k/R}(\xi)\Big|^2.
		\end{equation*}
		The supports of the summands $\sum_{\theta\subseteq\tau_k}\widehat{|f_{m,\theta}^{\mathcal{B}}|^2}(\xi)\eta_{\geq R_k/R}(\xi)$, ranging over distinct $\tau_k$, have greatest overlap on the circle of radius $R_k/R$, where the overlap is $O(N)$. By Cauchy-Schwarz,
		\begin{equation*}
			\int_{|\xi|\geq R_k/R}\Big|\sum_{\tau_k}\sum_{\theta\subseteq\tau_k}\widehat{|f_{m,\theta}^{\mathcal{B}}|^2}(\xi)\eta_{\geq R_k/R}(\xi)\Big|^2\lesssim (\log R)\sum_{\tau_k}\int_{|\xi|\geq R_k/R}\Big|\sum_{\theta\subseteq\tau_k}\widehat{|f_{m,\theta}^{\mathcal{B}}|^2}(\xi)\eta_{\geq R_k/R}(\xi)\Big|^2.
		\end{equation*}
		We conclude by enlarging the domain of integration on the right-hand side and using Plancherel.
		
		(b): By Plancherel,
		
		\begin{equation*}
			\int\Big|\sum_{\tau_k}|f_{m,\tau_k}^{\mathcal{B}}|^2*\eta_{\geq R_k^{-1}}^\vee\Big|^2=\int_{|\xi|\geq R_k^{-1}}\Big|\sum_{\tau_k}\widehat{|f_{m,\tau_k}^{\mathcal{B}}|^2}(\xi)\eta_{\geq R_k^{-1}}(\xi)\Big|^2.
		\end{equation*}
		Each $\widehat{|f_{m,\tau_k}^{\mathcal{B}}|^2}$ is supported in the set $(N-m+2)(\tau_k-\tau_k)\subseteq (\log R)(\tau_k-\tau_k)$, and the maximal overlap between these for distinct $\tau_k$ in the region $|\xi|\geq R_k^{-1}$ occurs when $|\xi|=R_k^{-1}$, where the overlap is $\lesssim\log R$. By Cauchy-Schwarz and Plancherel,
		\begin{equation*}
			\int_{|\xi|\geq R_k^{-1}}\Big|\sum_{\tau_k}\widehat{|f_{m,\tau_k}^{\mathcal{B}}|^2}(\xi)\eta_{\geq R_k^{-1}}(\xi)\Big|^2\lesssim\log R\sum_{\tau_k}\int\Big||f_{m,\tau_k}^{\mathcal{B}}|^2*\eta_{\geq R_k^{-1}}^\vee\Big|^2.
		\end{equation*}
		Lastly, $\|\eta_{\geq R_k^{-1}}^\vee\|_1=O(1)$ by a change-of-variable, thus
		\begin{equation*}
			\int\Big|\sum_{\tau_k}|f_{m,\tau_k}^{\mathcal{B}}|^2*\eta_{\geq R_k^{-1}}^\vee\Big|^2\lesssim(\log R)\sum_{\tau_k}\int|f_{m,\tau_k}^{\mathcal{B}}|^4
		\end{equation*}
		as claimed.
		
		(c): Reasoning as in (b), note that $\Big[f_{m,\tau_k}^\mathcal{B}\overline{f_{m,\tau_k'}^\mathcal{B}}\Big]*\eta_{\sim R_{k+\ell}^{-1}}^\vee$ has Fourier support in the set $(N-m+2)(\tau_k-\tau_k')$. Note that $\tau_k-\tau_k'$ is contained in a set of the form $(c_{\tau_k}-c_{\tau_k'})+C(\log R)(\tau_k-\tau_k)\subseteq C'(\log R)^2(\tau_k-\tau_k)$ (c.f. Remark \ref{nearcomp}). As this is the case for each $\tau_k'$ for which $\tau_k'\text{ near }\tau_k$, we conclude that $\sum_{\tau_k':\tau_k'\text{ near }\tau_k}\Big[f_{m,\tau_k}^\mathcal{B}\overline{f_{m,\tau_k'}^\mathcal{B}}\Big]*\eta_{\sim R_{k+\ell}^{-1}}^\vee$ has Fourier support in the set $C'(\log R)^2(\tau_k-\tau_k)$. On the circle of radius $R_{k+l}^{-1}$, the overlap between these sets is $O((\log R)^{2}R_\ell)$, so
		\begin{equation*}
			\int\Big|\sum_{\tau_{k}}\sum_{\tau_k'\text{ near }\tau_k}(f_{m,\tau_{k}}^\mathcal{B}\overline{f_{m,\tau_{k}'}^\mathcal{B}})*\eta_{\geq R_{k+\ell}^{-1}}^\vee\Big|^2\lesssim(\log R)^{2}R_\ell\sum_{\tau_k}\int\Big|\sum_{\tau_k'\text{ near }\tau_k}|f_{m,\tau_k}^\mathcal{B}|^2*\eta_{\geq R_{k+\ell}^{-1}}^\vee\Big|^2.
		\end{equation*}
		By Cauchy-Schwarz,
		\begin{equation*}
			\sum_{\tau_k}\int\Big|\sum_{\tau_k'\text{ near }\tau_k}|f_{m,\tau_k}^\mathcal{B}|^2*\eta_{\geq R_{k+\ell}^{-1}}^\vee\Big|^2\lesssim(\log R)\sum_{\tau_k}\int\Big|\,|f_{m,\tau_k}^\mathcal{B}|^2*\eta_{\geq R_{k+\ell}^{-1}}^\vee\Big|^2,
		\end{equation*}
		and since $\|\eta_{\geq R_{k+l}^{-1}}\|_1=O(1)$ we conclude that
		\begin{equation*}
			\int\Big|\sum_{\tau_{k}}\sum_{\tau_k'\text{ near }\tau_k}(f_{m,\tau_{k}}^\mathcal{B}\overline{f_{m,\tau_{k}'}^\mathcal{B}})*\eta_{\geq R_{k+\ell}^{-1}}^\vee\Big|^2\lesssim(\log R)^{3}R_\ell\sum_{\tau_k}\int|f_{m,\tau_k}^\mathcal{B}|^4
		\end{equation*}
		as claimed.
		
	\end{proof}
	
	\begin{lemma}[Weak high-domination of bad parts] Let $2\leq m\leq N$ and $0\leq k<m$.
		
		\begin{enumerate}
			\item[(a)] We have the estimate
			\begin{equation*}
				\Big|\sum_{\tau_{m-1}\subseteq\tau_k}|f_{m,\tau_{r}}^{\mathcal{B}}|^2*\eta_{\leq R_{m-1}/R}^\vee(x)\Big|\lesssim\frac{\alpha^2(\#\tau_{m-1}\subseteq\tau_k)}{C_{\mathfrak{p}}(\log R)^{6}(\#\tau_{m-1})^2}.
			\end{equation*}
			\item[(b)] Suppose $\alpha\lesssim(\log R)^{3}|f_{m,\tau_k}^\mathcal{B}(x)|$. Then
			\begin{equation*}
				\sum_{\tau_{m-1}\subseteq\tau_k}|f_{m,\tau_{m-1}}^{\mathcal{B}}|^2(x)\lesssim \Big|\sum_{\tau_{m-1}\subseteq\tau_k}|f_{m,\tau_{m-1}}^{\mathcal{B}}|^2*\eta_{\geq R_{m-1}/R}^\vee(x)\Big|.
			\end{equation*}
			
		\end{enumerate}
		
	\end{lemma}
	\begin{proof}
		(a): 
		By the low lemma,
		\begin{equation*}
			\sum_{\tau_{r}\subseteq\tau_k}|f_{m,\tau_{r}}^{\mathcal{B}}|^2*\eta_{\leq R_{m-1}/R}^\vee(x)=\sum_{\theta\subseteq\tau_k}\sum_{\theta'\text{ near }\theta}(f_{m,\theta}^{\mathcal{B}}\overline{f_{m,\theta}^{\mathcal{B}}})*\eta_{\leq R_{m-1}/R}^\vee(x).
		\end{equation*}
		By the definition of ``near,''
		\begin{equation*}
			\Big|\sum_{\theta\subseteq\tau_k}\sum_{\theta'\text{ near }\theta}(f_{m,\theta}^{\mathcal{B}}\overline{f_{m,\theta}^{\mathcal{B}}})*\eta_{\leq R_{m-1}/R}^\vee(x)\Big|\lesssim(\log R)\sum_{\theta\subseteq\tau_k}|f_{m,\theta}^{\mathcal{B}}|^2*|\eta_{\leq R_{m-1}/R}^\vee|(x).
		\end{equation*}
		By local constancy,
		\begin{equation*}
			\sum_{\theta\subseteq\tau_k}|f_{m,\theta}^{\mathcal{B}}|^2*|\eta_{\leq R_{m-1}/R}^\vee|(x)\lesssim\sum_{\theta\subseteq\tau_k}|f_{m,\theta}^{\mathcal{B}}|^2*|\rho_{(\log R)\theta}^\vee|*|\eta_{\leq R_{m-1}/R}^\vee|(x).
		\end{equation*}
		If $\theta\subseteq\tau_{m-1}$,
		\begin{equation*}
			|f_{m,\theta}^{\mathcal{B}}|^2*|\rho_{(\log R)\theta}^\vee|*|\eta_{\leq R_{m-1}/R}^\vee|(x)=\int\Big|\sum_{U\not\in\mathcal{G}_{\tau_{m-1}}}\psi_Uf_{m,\theta}\Big|^2(y)\Big(|\rho_{(\log R)\theta}^\vee|*|\eta_{\leq R_{m-1}/R}^\vee|\Big)(x-y)dy.
		\end{equation*}
		Since $\psi_U$ are all real and nonnegative,
		\begin{equation*}
			\begin{split}
				\int\Big|\sum_{U\not\in\mathcal{G}_{\tau_{m-1}}}&\psi_Uf_{m,\theta}\Big|^2(y)\Big(|\rho_{(\log R)\theta}^\vee|*|\eta_{\leq R_{m-1}/R}^\vee|\Big)(x-y)dy\\
				&=\int\sum_{U\not\in\mathcal{G}_{\tau_{m-1}}}\psi_U(y)|f_{m,\theta}|^2(y)\sum_{U'\not\in\mathcal{G}_{\tau_{m-1}}}\psi_{U'}(y)\Big(|\rho_{(\log R)\theta}^\vee|*|\eta_{\leq R_{m-1}/R}^\vee|\Big)(x-y)dy.
			\end{split}
		\end{equation*}
		Since $\{\psi_{U'}\}_{U'}$ form a partition of unity, $\sum_{U'\not\in\mathcal{G}_{\tau_{m-1}}}\psi_{U'}(y)\leq 1$, and so
		\begin{equation*}
			\begin{split}
				\int\sum_{U\not\in\mathcal{G}_{\tau_{m-1}}}&\psi_U(y)|f_{m,\theta}|^2(y)\sum_{U'\not\in\mathcal{G}_{\tau_{m-1}}}\psi_{U'}(y)\Big(|\rho_{(\log R)\theta}^\vee|*|\eta_{\leq R_{m-1}/R}^\vee|\Big)(x-y)dy\\
				&\leq\int\sum_{U\not\in\mathcal{G}_{\tau_{m-1}}}\psi_U(y)|f_{m,\theta}|^2(y)\Big(|\rho_{(\log R)\theta}^\vee|*|\eta_{\leq R_{m-1}/R}^\vee|\Big)(x-y)dy,
			\end{split}
		\end{equation*}
		which by H\"older is bounded from above by
		\begin{equation*}
			\sum_{U\not\in\mathcal{G}_{\tau_{m-1}}}\Big\|\psi_U^{1/2}(y)\Big(|\rho_{(\log R)\theta}^\vee|*|\eta_{\leq R_m/R}^\vee|\Big)(x-y)\Big\|_{L_y^\infty}\int\psi_U^{1/2}|f_{m,\theta}|^2(y)dy.
		\end{equation*}
		Note that, for each $x$, the function $y\mapsto |\rho_{(\log R)\theta}^\vee|*|\eta_{\leq R_m/R}^\vee|(x-y)$ is approximately constant on rectangles of dimensions $\sim (\log R)^{-1}R\times R/R_m$, with long edge parallel to $\mathbf{n}_{c_\theta}$. By rapid decay of $\psi_U$ outside of $U$,
		\begin{equation*}
			\begin{split}
				\sum_{U\|U_{\tau_{m-1},R}}\Big\|\psi_U^{1/2}(y)&\Big(|\rho_{(\log R)\theta}^\vee|*|\eta_{\leq R_m/R}^\vee|\Big)(x-y)\Big\|_{L_y^\infty}\lesssim\sum_{U\|U_{\tau_{m-1},R}}\Big\||\rho_{(\log R)\theta}^\vee|*|\eta_{\leq R_m/R}^\vee|(x-y)\Big\|_{L_y^\infty(U)}\\
				&\leq\sum_{U\|U_{\tau_{m-1},R}}\sum_{\substack{V\sim(\log R)^{-1}R\times R/R_m\\V\subseteq U}}\Big\||\rho_{(\log R)\theta}^\vee|*|\eta_{\leq R_m/R}^\vee|(x-y)\Big\|_{L_y^\infty(V)}\\
				&\lesssim\sum_{U\|U_{\tau_{m-1},R}}\sum_{\substack{V\sim(\log R)^{-1}R\times R/R_m\\V\subseteq U}}|V|^{-1}\Big\||\rho_{(\log R)\theta}^\vee|*|\eta_{\leq R_m/R}^\vee|(x-y)\Big\|_{L_y^1(V)}\\
				&=(\log R)|U|^{-1}\Big\||\rho_{(\log R)\theta}^\vee|*|\eta_{\leq R_m/R}^\vee|(x-y)\Big\|_{L_y^1(\R^2)}\\
				&\lesssim(\log R)|U|^{-1}.
			\end{split}
		\end{equation*}
		
		Additionally, the polynomial decay of $W_U$ allows us to take $\psi_U^{1/2}\lesssim W_U$, so in total we get
		\begin{equation*}
			|f_{m,\theta}^{\mathcal{B}}|^2*|\rho_{(\log R)\theta}^\vee|*|\eta_{\leq R_{m-1}/R}^\vee|(x)\lesssim(\log R)\max_{U\not\in\mathcal{G}_{\tau_{m-1}}}\strokedint_U|f_{m,\theta}|^2(y)dy.
		\end{equation*}
		If we sum over all $\theta\subseteq\tau_{m-1}$, and use the hypothesis $U\not\in\mathcal{G}_{\tau_{m-1}}$, we see that
		\begin{equation*}
			\begin{split}
				\sum_{\tau_{m-1}\subseteq\tau_k}\sum_{\theta\subseteq\tau_{m-1}}&|f_{m,\theta}^{\mathcal{B}}|^2*|\rho_{(\log R)\theta}^\vee|*|\eta_{\leq R_{m-1}/R}^\vee|(x)\\
				&\lesssim(\log R)^2\sum_{\tau_{m-1}\subseteq\tau_k}\sup_{U\not\in\mathcal{G}_{\tau_{m-1}}}\strokedint_U\sum_{\theta\subseteq\tau_{m-1}}|f_{m,\theta}|^2(y)dy\\
				&\leq\frac{\alpha^2(\#\tau_{m-1}\subseteq\tau_k)}{(\#\tau_{m-1})^2}\frac{1}{C_{\mathfrak{p}}(\log R)^{6}}.
			\end{split}
		\end{equation*}
		
		(b): Write $f_{m,\tau_k}^{\mathcal{B}}=\sum_{\tau_{m-1}\subseteq\tau_k}f_{m,\tau_{m-1}}^{\mathcal{B}}$, where $f_{m,\tau_{m-1}}^{\mathcal{B}}=\sum_{U\not\in\mathcal{G}_{\tau_{m-1}}}\psi_U\sum_{\theta\subseteq\tau_{m-1}}f_{m,\theta}$. By Cauchy-Schwarz,
		\begin{equation*}
			\alpha\lesssim(\#\tau_{m-1}\subseteq\tau_k)^{1/2}(\log R)^{3}\left(\sum_{\tau_{m-1}\subseteq\tau_k}|f_{m,\tau_{m-1}}^{\mathcal{B}}|^2(x)\right)^{1/2}.
		\end{equation*}
		We assume for the sake of contradiction that
		\begin{equation*}
			\sum_{\tau_{m-1}\subseteq\tau_k}|f_{m,\tau_{m-1}}^{\mathcal{B}}|^2(x)\leq C_{\mathfrak{p}}^{1/2}\Big|\sum_{\tau_{m-1}\subseteq\tau_k}|f_{m,\tau_{m-1}}^{\mathcal{B}}|^2*\eta_{\leq R_{m-1}/R}^\vee(x)\Big|.
		\end{equation*}
		By (a),
		\begin{equation*}
			\sum_{\tau_{m-1}\subseteq\tau_k}|f_{m,\tau_{m-1}}^{\mathcal{B}}|^2(x)\lesssim C_{\mathfrak{p}}^{1/2}\frac{\alpha^2(\#\tau_{m-1}\subseteq\tau_k)}{(\#\tau_{m-1})^2C_{\mathfrak{p}}(\log R)^{6}}.
		\end{equation*}
		On the other hand, we assumed the estimate
		\begin{equation*}
			\alpha^2\lesssim (\#\tau_{m-1}\subseteq\tau_k)(\log R)^{6}\sum_{\tau_{m-1}\subseteq\tau_k}|f_{m,\tau_{m-1}}^\mathcal{B}|^2(x),
		\end{equation*}
		so that
		\begin{equation*}
			\sum_{\tau_{m-1}\subseteq\tau_k}|f_{m,\tau_{m-1}}^{\mathcal{B}}|^2(x)\lesssim C_{\mathfrak{p}}^{1/2}\frac{(\log R)^{6}}{C_{\mathfrak{p}}(\log R)^{6}}\frac{(\#\tau_{m-1}\subseteq\tau_k)^2}{(\#\tau_{m-1})^2}\sum_{\tau_{m-1}\subseteq\tau_k}|f_{m,\tau_{m-1}}^\mathcal{B}|^2(x),
		\end{equation*}
		i.e.
		\begin{equation*}
			C_{\mathfrak{p}}^{1/2}\lesssim 1.
		\end{equation*}
		If $C_{\mathfrak{p}}$ is chosen as a sufficiently large universal constant (i.e. independently of $f,R$), then we conclude by contradiction that
		\begin{equation*}
			\sum_{\tau_{m-1}\subseteq\tau_k}|f_{m,\tau_{m-1}}^{\mathcal{B}}|^2(x)> C_{\mathfrak{p}}^{1/2}\Big|\sum_{\tau_{m-1}\subseteq\tau_k}|f_{m,\tau_{m-1}}^{\mathcal{B}}|^2*\eta_{\geq R_{m-1}/R}^\vee(x)\Big|,
		\end{equation*}
		i.e.
		\begin{equation*}
			\sum_{\tau_{m-1}\subseteq\tau_k}|f_{m,\tau_{m-1}}^{\mathcal{B}}|^2(x)\leq \frac{C_{\mathfrak{p}}^{1/2}}{C_{\mathfrak{p}}^{1/2}-1}\Big|\sum_{\tau_{m-1}\subseteq\tau_k}|f_{m,\tau_{m-1}}^{\mathcal{B}}|^2*\eta_{\leq R_{m-1}/R}^\vee(x)\Big|.
		\end{equation*}
		Since $C_{\mathfrak{p}}$ is chosen to be a large constant, we conclude that the prefactor $\frac{C_{\mathfrak{p}}^{1/2}}{C_{\mathfrak{p}}^{1/2}-1}$ is $O(1)$, so we are done.
	\end{proof}

	\printbibliography
	
\end{document}